%% file: main_iop_r1.tex
\documentclass[12pt]{iopart}

\usepackage{algorithm}
\usepackage{algorithmic} 

\usepackage{tikz}\usetikzlibrary{arrows, decorations.pathreplacing}

\usepackage{iopams}

\expandafter\let\csname equation*\endcsname\relax

\expandafter\let\csname endequation*\endcsname\relax

\usepackage{amsmath}
\usepackage{amsopn}
\usepackage{amsthm}
\usepackage{mathtools}
\usepackage{graphicx}
\usepackage{array}

\usepackage{hyperref}
\usepackage[noabbrev]{cleveref} 

\crefformat{equation}{(#2#1#3)}
\newtheorem{theorem}{Theorem}[section]
\newtheorem{lemma}[theorem]{Lemma}
\newtheorem{remark}[theorem]{Remark}
\newcommand{\mbbR}{\mathbb{R}}

\DeclareMathOperator*{\argmin}{arg\, min}

\begin{document}

\title[Projected Newton Method]{Projected Newton Method for noise constrained Tikhonov regularization}

\author{J Cornelis, N Schenkels \& W Vanroose}

\address{Applied Mathematics Group, Department of Mathematics and Computer Science, University of Antwerp, Building G, Middelheimlaan 1, 2020 Antwerp, Belgium}
\ead{jeffrey.cornelis@uantwerp.be}
\vspace{10pt}
\begin{indented}
\item[]\today
\end{indented}

\begin{abstract}
Tikhonov regularization is a popular approach to obtain a meaningful solution for ill-conditioned linear least squares problems. A relatively simple way of choosing a good regularization parameter is given by Morozov's discrepancy principle. However, most approaches require the solution of the Tikhonov problem for many different values of the regularization parameter, which is computationally demanding for large scale problems. We propose a new and efficient algorithm which simultaneously solves the Tikhonov problem and finds the corresponding regularization parameter such that the discrepancy principle is satisfied. We achieve this by formulating the problem as a nonlinear system of equations and solving this system using a line search method. We obtain a good search direction by projecting the problem onto a low dimensional Krylov subspace and computing the Newton direction for the projected problem. This \textit{projected Newton direction}, which is significantly less computationally expensive to calculate than the true Newton direction, is then combined with a backtracking line search to obtain a globally convergent algorithm, which we refer to as \textit{the Projected Newton method}. We prove convergence of the algorithm and illustrate the improved performance over current state-of-the-art solvers with some numerical experiments. 

\end{abstract}

%
\vspace{2pc}


\section{Introduction}
Large scale ill-posed linear inverse problems of the form
$Ax = b$ with $A\in\mathbb{R}^{m\times n}, x \in \mathbb{R}^n$ and $b\in \mathbb{R}^m$ where $m\geq n$ arise in countless scientific and industrial applications. The singular values of such a matrix typically decay to zero, which means the condition number $\kappa(A)$ is very large. This in turn means that small perturbations in the right-hand side $b$ can cause huge changes in the solution $x$. The right-hand side $b$ is generally the
perturbed version of the unknown exact measurements or observations $b_{ex} = b + e$. Thus we know that for ill-posed problems some form of
regularization has to be used in order to deal with the noise $e$ in the data $b$ and to
find a good approximation for the true solution $x_{ex}$ of $Ax = b_{ex}$ . One of the most widely
used methods to do so is Tikhonov regularization \cite{tikhonov1977solutions}. In its standard from, the Tikhonov
solution to the inverse problem is given by
\begin{equation}\label{eq:min_tikhonov}
x_{\alpha} = \argmin_{x\in\mathbb{R}^n} \frac{1}{2}||Ax-b||^2 + \frac{\alpha}{2} ||x||^2
\end{equation}
where $\alpha > 0$ is the regularization parameter and $||\cdot||$ is the standard Euclidean norm. By taking the gradient of the objective function of \cref{eq:min_tikhonov} and equating it to zero, it follows that the Tikhonov solution can alternatively be characterized as the solution of the linear system of normal equations
\begin{equation}\label{eq:lin_tikhonov}
\left(A^T A + \alpha I\right) x_{\alpha} =  A^T b
\end{equation}
where $I$ is the $n\times n$ identity matrix \cite{hansen2010}. The choice of the regularization parameter is very important since its value has
a significant impact on the quality of the reconstruction. If, on the one hand, $\alpha$ is chosen too large,
focus lies on minimizing the regularization term $||x||^2$ in \cref{eq:min_tikhonov}. The corresponding reconstruction $x_{\alpha}$ will therefore no longer be a good solution for the linear system $Ax = b$, will
typically have lost many details and be what is referred to as ``oversmoothed''. If, on
the other hand,  $\alpha$ is chosen too small, focus lies on minimizing the residual $||Ax-b||^2 $.
This, however, means that the measurement errors $e$ are not suppressed and that the reconstruction
$x_{\alpha}$ will be ``overfitted'' to the measurements. 

One way of choosing the regularization parameter is the L-curve method. If $x_{\alpha}$ is
the solution of the Tikhonov problem \cref{eq:min_tikhonov}, then the curve $(\log ||Ax_\alpha-b||, \log||x_\alpha||)$ typically has a rough ``L'' shape, see \cref{fig:curves}. Heuristically, the value for the regularization
parameter corresponding to the corner of this ``L'' has been proposed as a good regularization parameter because it balances model fidelity (minimizing the residual)
and regularizing the solution (minimizing the regularization term) \cite{calvetti1999, hansen1992, hansen1993, hansen2010}. The
problem with this method is that in order to find this value, the linear system \cref{eq:lin_tikhonov} has
to be solved for many different values of $\alpha$, which can be computationally expensive
and inefficient for large scale problems. Moreover, it is a heuristic parameter choice method, meaning that it does not yield convergence to $x_{ex}$ when the noise $e$ in the data tends to zero, so one should be cautious when using it \cite{engl1996regularization}. However, the strength of the parameter choice method lies in the fact that it does not require any knowledge of the size of the noise $e$. 

\begin{figure}
		\centering
		\input{curves}
		\caption{Sketch of the L-curve (left) and the D-curve (right). The value
		for $\alpha$ proposed by the L-curve method is typically slightly larger
		than the one proposed by the discrepancy principle \cite{hansen1992}.}
		\label{fig:curves}
\end{figure}
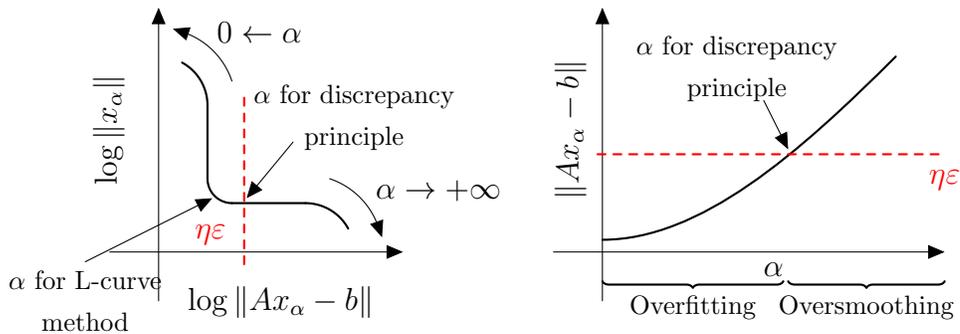

Another way of choosing the regularization parameter is Morozov's discrepancy
principle \cite{morozov1984}. Here, the regularization parameter is chosen such that
\begin{equation}\label{eq:discrepancy_principle}
||Ax_\alpha-b|| = \eta \epsilon 
\end{equation}
with $\epsilon = ||e||$ the size of the measurement error and $1 \leq \eta$ a tolerance value. The idea behind this
choice is that finding a solution $x_\alpha$ with a lower residual can only lead to overfitting.
Similarly to the L-curve, we can look at the curve $(\alpha,||Ax_\alpha - b||)$, which is often referred to as the discrepancy curve or D-curve, see \cref{fig:curves}. Note that in practice we never have the exact value $\epsilon$, so this approach assumes we have a good estimate for the ``noise-level'' of the inverse problem. We could use a root-finding method like the secant or regula falsi method \cite{allen2011numerical,burden2000numerical} to find the regularization parameter $\alpha$ which satisfies \cref{eq:discrepancy_principle}, but this would again require us to solve the linear system \cref{eq:lin_tikhonov} multiple times.

In this paper we develop a new and efficient algorithm, which we call \textit{the Projected Newton method}, that simultaneously updates the
solution $x$ and the regularization parameter $\alpha$ such that the Tikhonov
equation \cref{eq:lin_tikhonov} and Morozov's discrepancy principle \cref{eq:discrepancy_principle}
are both satisfied. We start by considering an equivalent formulation as a constrained optimization problem, which we refer to as \textit{the noise constrained Tikhonov problem}. By taking the gradient of the Lagrangian of the optimization problem, this formulation naturally leads to a nonlinear system of equations that can be solved using a Newton type method, which is precisely the approach taken in \cite{landi2008lagrange}. For large scale problems, however, computing the Newton direction is computationally demanding. We therefore propose to project the noise constrained Tikhonov problem onto a low dimensional Krylov subspace. In each iteration of the Projected Newton method we first expand the Krylov subspace with one dimension and then compute the Newton direction for the projected problem, which is much cheaper to compute than the actual Newton direction. 
This \textit{projected Newton direction} is then combined with a backtracking line search to obtain a robust and globally convergent algorithm. 

The rest of the paper is organized as follows. \Cref{sec:newtons_method} introduces the noise constrained Tikhonov problem and reviews some basic properties of Newton's method and the approach taken in \cite{landi2008lagrange}. The main novel contribution of this work is presented in \cref{sec:projected_newton_method}, where we derive the Projected Newton method. In \cref{sec:proofconv} a convergence result is formulated and proven for the new algorithm. Experimental
results, as well as a reference method \cite{gazzola2014_2}, can be found in \cref{sec:numex}. Lastly, this work is concluded and possible future research directions are proposed in \cref{sec:conclusion}.

\section{Newton's method for noise constrained Tikhonov regularization} \label{sec:newtons_method}

Let us consider the following equality constrained optimization problem:
\begin{equation} \label{eq:noise_constrained}
\min_{x\in\mathbb{R}^n} \hspace{0.2cm} \frac{1}{2}||x||^2 \hspace{0.5cm} \text{subject to} \hspace{0.5cm}  \frac{1}{2}||Ax-b||^2 = \frac{\sigma^2}{2}
\end{equation}
where we denote $\sigma = \eta\varepsilon$ the value used in the discrepancy principle, see \cref{eq:discrepancy_principle}. The Lagrangian of \cref{eq:noise_constrained} is given by
\begin{equation} \label{eq:lagrangian}
\mathcal{L}(x,\lambda) = \frac{1}{2}||x||^2 + \lambda \left(  \frac{1}{2}||Ax-b||^2 - \frac{\sigma^2}{2} \right)
\end{equation}
where $\lambda \in \mathbb{R}$ is the Lagrange multiplier. It follows from classical constrained optimization theory that if $x^{*}$ is a solution of \cref{eq:noise_constrained} then there exists a Lagrange multiplier $\lambda^{*}$ such that $F(x^{*},\lambda^{*}) = 0$ with
\begin{equation} \label{eq:F}
F(x,\lambda) = \begin{pmatrix}
 \lambda A^T (Ax - b) + x \\
\frac{1}{2} ||Ax-b||^2 - \frac{\sigma^2}{2} 
\end{pmatrix}.
\end{equation}
The first component of this function, which is nonlinear jointly in $x$ and $\lambda$, is the gradient of the Lagrangian with respect to $x$ , i.e. $\nabla_x \mathcal{L}(x,\lambda) $, while the second component is simply the constraint. These are in fact the first order optimality conditions, also known as Karush-Kuhn-Tucker or KKT conditions \cite{nocedal2006numerical}. Note that any point $(x,\lambda)$ with $\lambda > 0$ that is a root of the first component of \cref{eq:F} is a Tikhonov solution \cref{eq:lin_tikhonov} with $\alpha = 1/\lambda$:
\begin{eqnarray} 
\label{eq:transf}
\lambda A^T (Ax - b) + x = 0 &\Leftrightarrow&  A^T (Ax - b) + \frac{1}{\lambda} x = 0 \nonumber\\
&\Leftrightarrow& A^T (Ax - b) +\alpha x = 0 \\ 
&\Leftrightarrow& (A^T A + \alpha I )x = A^T b.\nonumber
\end{eqnarray} 
It is a well known result that \cref{eq:noise_constrained} has a unique solution $x^{*}$ and that the corresponding Lagrange multiplier $\lambda^{*}$ is strictly positive \cite{engl1996regularization}. Now if we write $\alpha^* = 1/\lambda^*$ it follows from the discussion above that $(x^{*},\alpha^{*})$ satisfies the Tikhonov equations \cref{eq:lin_tikhonov} as well as the discrepancy principle \cref{eq:discrepancy_principle}. This means that if we solve \cref{eq:noise_constrained}, we have simultaneously found the regularization parameter and corresponding Tikhonov solution that satisfies the discrepancy principle. Henceforth we refer to \cref{eq:noise_constrained} as \textit{the noise constrained Tikhonov problem}. 

A modification of Newton's method to solve the nonlinear system of equations $F(x,\lambda) = 0$ is presented in \cite{landi2008lagrange}, which the author refers to as \textit{the Lagrange Method} since it is based on the Lagrangian \cref{eq:lagrangian}. The Newton direction for this particular problem starting from a point $(x,\lambda)$ is defined as the solution of the linear system
\begin{equation}\label{eq:newton_system}
J(x,\lambda)\begin{pmatrix} \Delta x \\ \Delta \lambda \end{pmatrix} = - F(x,\lambda)
\end{equation}
where $J(x,\lambda) \in \mathbb{R}^{(n+1)\times (n+1)}$ is the Jacobian matrix of $F(x,\lambda)$ and is given by 
\begin{equation} \label{eq:jacf}
J(x,\lambda) =\begin{pmatrix} \lambda A^T A + I &  A^T(Ax - b) \\ (Ax-b)^T A & 0 \end{pmatrix}. 
\end{equation}
We can express the determinant of this matrix as
\begin{equation} \label{eq:determinant}
\det J(x,\lambda) = -(Ax-b)^T A(\lambda A^T A + I)^{-1} A^T(Ax - b) \det(\lambda A^TA + I). 
\end{equation} 
due to the $2\times 2$ block structure of the matrix \cite{zhang2006schur}. Thus we know that $J(x,\lambda)$ is nonsingular when $\lambda\geq0$ and $A^T(Ax - b) \neq 0$. Indeed, this follows from the fact that both $\lambda A^T A + I$ and $(\lambda A^T A + I)^{-1}$ are positive definite and hence $\det J(x,\lambda) < 0$.

It is well known that the Newton direction is a descent direction for the merit function $f(x,\lambda) = \frac{1}{2} ||F(x,\lambda)||^2$, which can be seen from the following calculation
\begin{eqnarray*}
\nabla f(x,\lambda)^T \begin{pmatrix}\Delta x \\ \Delta \lambda \end{pmatrix} &=& \left(J(x,\lambda)^T F(x,\lambda)\right)^T  \left(- J(x,\lambda)^{-1}F(x,\lambda)\right) \\
&=& - ||F(x,\lambda)||^2 \leq 0.
\end{eqnarray*} 
Now it follows from Taylor's theorem \cite{nocedal2006numerical} that starting from the point $(x,\lambda)$, we can either find a step-length $\gamma > 0$ such that $f(x + \gamma \Delta x, \lambda + \gamma \Delta \lambda) < f(x, \lambda)$ or we have found the solution to $F(x,\lambda) = 0$. In \cite{landi2008lagrange} a different merit function is chosen, namely 
\begin{equation} \label{eq:merit}
m(x,\lambda) = \frac{1}{2}|| \lambda A^T(Ax - b) + x ||^2 + \frac{w}{2} \left(\frac{1}{2} ||Ax-b||^2 - \frac{\sigma^2}{2}  \right)^2
\end{equation} 
 with $w\in\mathbb{R}$ a fixed weight. Note that for $w = 1$ this merit function coincides with the usual merit function $f(x,\lambda)$. The Lagrange method calculates in each iteration an approximate solution $\tilde{\Delta}$ to \cref{eq:newton_system} using the Krylov subspace method MINRES \cite{paige1975solution} and then chooses the step-length $\gamma$ such that there is a sufficient decrease of the merit function $m(x,\lambda)$, i.e. 
\begin{equation*}
m(x + \gamma \tilde{\Delta} x, \lambda + \gamma \tilde{\Delta} \lambda)  < m(x,\lambda) + c \gamma \nabla m(x,\lambda)^T \tilde{\Delta} 
\end{equation*} 
with $c\in (0,1)$ and such that $\lambda> 0 $ and $A^T(Ax - b) \neq 0$ by means of a backtracking line search. Furthermore it is shown that by choosing a specific tolerance $\xi$ for the relative residual $||J(x,\lambda)\tilde{\Delta} + F(x,\lambda)|| \leq \xi ||F(x,\lambda)|| $ such a step-length $\gamma$ can always be found and that the algorithm converges to the unique solution of the noise constrained Tikhonov problem \cref{eq:noise_constrained}. Even when solving \cref{eq:newton_system} to a certain tolerance with a Krylov subspace method, calculating the Newton direction is computationally expensive for large scale problems since each MINRES iteration requires a matrix-vector product with $A$ and $A^T$.  

\begin{remark}
The Lagrange method is able to solve the more general equality constrained optimization problem 
\begin{equation}\label{eq:general_reg}
\min_{x\in\mathbb{R}^n} \hspace{0.2cm} \phi(x) \hspace{0.5cm} \text{subject to} \hspace{0.5cm}  \frac{1}{2}||Ax-b||^2 = \frac{\sigma^2}{2}
\end{equation}
where $\phi(x)$ is a general regularization functional. However, since we are concerned with the Tikhonov solution,  we only consider the functional $\phi(x) = ||x||^2/2$. 
\end{remark}

\section{Projected Newton method}\label{sec:projected_newton_method}
In this section we derive the Projected Newton method by projecting the noise constrained Tikhonov problem \cref{eq:noise_constrained} onto a $k$ dimensional Krylov subspace $\mathcal{K}_k(A^TA,A^Tb)$, where $k$ is the iteration index of the new algorithm. In each iteration we expand the basis $V_{k}$ for the Krylov subspace by the bidiagonalization algorithm due to Golub and Kahan \cite{golub1965} and calculate an approximate solution in the Krylov subspace using a single Newton iteration on the projected system. We show that by calculating the Newton direction in the projected space, which can be done very efficiently for small values of $k$, we get a descent direction for the merit function $f(x,\lambda)$. We start our discussion by a brief review of the bidiagonal decomposition \cite{golub1965}.
\subsection{Bidiagonalization}
\begin{theorem}[Bidiagonal decomposition]
If $A\in\mathbb{R}^{m \times n}$ with $m\geq n$, then there exist orthonormal matrices 
	\[
				U = (u_0, u_1, \ldots, u_{m - 1})\in\mbbR^{m\times m}\quad\text{and}\quad
						V = (v_0, v_1, \ldots, v_{n-1})\in\mbbR^{n\times n}
		\]
		and a lower bidiagonal matrix $B$ with dimensions $(n+1)\times n$ when $m>n$ and $n\times n$ when $m=n$
		such that
		\[
				A = U\begin{pmatrix}B\\0\end{pmatrix}V^T.
		\]
\end{theorem}
\begin{proof}
See \cite{golub1965}.
\end{proof}
Starting from a given unit vector $u_0\in\mathbb{R}^m$ it is possible to
generate the columns of $U$, $V$ and $B$ recursively using the Bidiag1 procedure
proposed by Paige and Saunders \cite{paige1982_2, paige1982}, see \cref{alg:bidiag1}. Reorthogonalization can be added for numerical stability.
Note that this bidiagonal decomposition is the basis for the LSQR algorithm and
that after $k$ steps of Bidiag1 starting with the initial vector $u_0 = b/\left\|b\right\|$
we have matrices $V_k = [v_0,v_1,\ldots,v_{k-1}]\in\mbbR^{n\times k}$ and $U_{k + 1}=[u_0,u_1,\ldots,u_k]\in\mbbR^{m\times(k + 1)}$
with orthonormal columns and a lower bidiagonal matrix $B_{k + 1,k}\in\mbbR^{(k + 1)\times k}$
that satisfy
\begin{align} 
AV_{k} &= U_{k+1}B_{k+1,k} \label{eq:avub} \\ 
A^T U_{k+1} &= V_{k}B_{k+1,k}^T + \mu_{k}v_{k} e_{k + 1}^T  \label{eq:atu}
\end{align}
with $e_{k + 1} = (0,0,\ldots,0,1)^T \in\mathbb{R}^{k + 1}$. Here, the matrix $B_{k+1,k}$ consists of the first $k+1$ rows and $k$ columns of $B$, i.e we denote:
		\[
				B_{k+1,k}=\begin{pmatrix}\mu_0\\ \nu_1 & \mu_1\\  & \nu_2 & \ddots\\  &  &
						\ddots & \mu_{k-1}\\ 
						&  &  & \nu_{k}\end{pmatrix}\in\mathbb{R}^{(k+1)\times k}.
		\]
		Note that in the case of a square matrix $A$ the coefficient $\nu_n=0$ vanishes. 

The columns of $V_{k}$ and $U_{k}$ both span a Krylov subspace of dimension $k$, more specifically we have
\begin{equation*}
\begin{cases}
 \mathcal{R}(V_k) = \mathcal{K}_k(A^TA, A^Tb) \\
  \mathcal{R}(U_k)= \mathcal{K}_k(AA^T, b) 
\end{cases}
\end{equation*}
where for a general square matrix $M$ and vector $z$ the Krylov subspace of dimension $k$ is defined as
\begin{equation*}
 \mathcal{K}_k(M,z) = \mathcal{R}\left([z,Mz,M^2z, \ldots, M^{k-1}z]\right)
\end{equation*}
and $\mathcal{R}(\cdot)$ denotes the range of the matrix, i.e. the span of its column vectors.
\begin{algorithm}
		\caption{Bidiag1}\label{alg:bidiag1}
		\begin{algorithmic}[1]
				\STATE{ Set initial unit vector $u_0= b/\left\|b\right\|$.}
				\STATE{ Set $\nu_0 v_{-1} = 0$}
				\FOR{$k = 0, 1, \ldots, n - 1$}
						\STATE{ $r_k = A^Tu_k - \nu_kv_{k-1}$}
						\STATE{ Reorthogonalize $r_k$ w.r.t. $V_k$ if necessary.}
						\STATE{ $\mu_k = \left\|r_k\right\|$ and $v_k = r_k/\mu_k$.}
						\STATE{ $p_k = Av_k - \mu_ku_k$}
						\STATE{ Reorthogonalize $p_k$ w.r.t. $U_{k+1}$ if necessary.}
						\STATE{ $\nu_{k+1} = \left\|p_k\right\|$ and $u_{k+1} = p_k/\nu_{k+1}$.}
				\ENDFOR
		\end{algorithmic}
\end{algorithm}
\begin{remark} \label{thm:breakdown}
It is possible that a breakdown may occur in \Cref{alg:bidiag1} in iteration $k$ when $\mu_k = 0$ on line $6$ or $v_{k+1}=0$ on line $9$. This will occur when the vector $b$ is a linear combination of less than $k$ eigenvectors of $AA^T$. However, since we assume that $b$ contains a random noise component $e$, the probability that this will happen is extremely low. Moreover, a breakdown also implies that we have found an invariant subspace. In the case of the LSQR algorithm, a breakdown means that the exact solution of the least squares problem $\min||Ax-b||$ has been found \cite{jia2010some}. In the following we will assume that no such breakdown occurs.
\end{remark}
\subsection{The projected Newton direction}
In order to solve the noise constrained Tikhonov problem \cref{eq:noise_constrained} we calculate a series of approximate solutions in the Krylov subspace spanned by the columns of the orthonormal basis $V_k$: 
\begin{equation*}
		x_k\in\mathcal{R}(V_k) = \mathcal{K}_k(A^TA, A^Tb) = \mathcal{R}\left(\left[ A^T b, (A^T A ) A^Tb,\ldots, (A^T A)^{k-1} A^T b\right]\right).
\end{equation*}
This means that $x_k = V_k y_k$ for some $y_k \in \mathbb{R}^k$ with $k\geq 1$ and let $x_0 = 0$. Using \cref{eq:avub} and the fact that $V_k$ and $U_{k+1}$ have orthonormal columns we have $||x_k||  = ||V_k y_k || = ||y_k||$ and
 \begin{equation} \label{eq:norm_equal}
||Ax_k - b || = ||AV_{k} y_k - b|| = ||U_{k + 1}B_{k+1,k} y_k - U_{k+1}c_{k+1}|| =  ||B_{k+1,k} y_k -c_{k+1}|| 
\end{equation}
where $c_{k + 1} = (||b||,0,\ldots,0)^T \in\mathbb{R}^{k+1}$. 
So now we can see that the \textit{projected minimization problem} of dimension $k$: 
\begin{equation} \label{projected_minequations1}
\min_{x_k\in\mathcal{R}(V_{k})} \hspace{0.2cm} \frac{1}{2}||x_k||^2 \hspace{0.5cm} \text{subject to} \hspace{0.5cm} \frac{1}{2} ||Ax_k-b||^2 = \frac{\sigma^2}{2}
\end{equation}
can alternatively be written using the bidiagonal decomposition as 
\begin{equation} \label{projected_minequations2}
\min_{y_k\in \mathbb{R}^k} \hspace{0.2cm} \frac{1}{2}||y_k||^2 \hspace{0.5cm} \text{subject to} \hspace{0.5cm} \frac{1}{2} ||B_{k+1,k} y_k -c_{k+1}|| ^2 =\frac{\sigma^2}{2}.
\end{equation}
Similarly as explained in \cref{sec:newtons_method} the solution $y_k^{*}$ of \cref{projected_minequations2} with corresponding Lagrange multiplier $\lambda_k^{*}$ satisfies $F^{(k)}(y_k^*,\lambda_k^*) = 0$ where 
the equation is now given by
\begin{equation} \label{eq:projected_f}
F^{(k)}(y,\lambda) = 
 \begin{pmatrix}\lambda B_{k + 1,k}^T (B_{k + 1,k} y - c_{k + 1}) + y \\
 \frac{1}{2} ||B_{k + 1,k} y - c_{k + 1}||^2 - \frac{\sigma^2}{2}
\end{pmatrix}. 
\end{equation}
The function $F^{(k)}(y,\lambda)$ can be seen as a projected version of the function $F(x,\lambda)$ given by \cref{eq:F}. The Jacobian $J^{(k)}(y,\lambda) \in \mathbb{R}^{(k + 1)\times(k + 1)}$ of $F^{(k)}(y,\lambda)$ is given by
\begin{equation*}
J^{(k)}(y,\lambda)  = \begin{pmatrix} \lambda B_{k + 1,k}^T B_{k + 1,k} + I_{k} &  B_{k + 1,k}^T( B_{k + 1,k} y - c_{k + 1}) \\  ( B_{k + 1,k} y - c_{k + 1})^T B_{k + 1,k} & 0 \end{pmatrix}
\end{equation*}
where $I_k$ is the $k\times k$ identity matrix. If $\lambda \geq 0$  then $\lambda B_{k + 1,k}^T B_{k + 1,k} + I_{k}$ is a positive definite matrix and we again have that $J^{(k)}(y,\lambda)$ is nonsingular if and only if  
\begin{equation} \label{eq:noteq}
B_{k + 1,k}^T( B_{k + 1,k} y - c_{k+1}) \neq 0.
\end{equation}
Let us denote $\bar{y}_{k - 1} = (y_{k - 1}^T,0)^T \in \mathbb{R}^{k}$ for $k\geq 1$ where $y_{k-1}$ is an approximation of the solution of the projected minimization problem of dimension $k-1$ and $y_0 = (\hspace{0.1cm})$ an empty vector and $\lambda_0$ an initial guess for the regularization parameter. Note that $\bar{y}_{k-1}$ can be seen as a good initial guess for the projected  minimization problem of dimension $k$, see \cref{projected_minequations2}. 
If $J^{(k)}(\bar{y}_{k - 1},\lambda_{k - 1})$ is nonsingular -- a condition we will enforce by step-length selection -- we can calculate the Newton direction for the projected function $F^{(k)}(y,\lambda)$ starting from the point $(\bar{y}_{k-1},\lambda_{k-1})$, i.e.: 
\begin{equation} \label{eq:newtoneq}
\begin{pmatrix} \Delta y_{k}\\ \Delta \lambda_{k} \end{pmatrix} = -  J^{(k)}(\bar{y}_{k-1},\lambda_{k-1}) ^{-1} F^{(k)}(\bar{y}_{k-1},\lambda_{k-1})
\end{equation}
We can then update $\bar{y}_{k-1}$ and $\lambda_{k-1}$ by 
\begin{equation*}
\begin{cases}
y_{k} = \bar{y}_{k-1} + \gamma  \Delta y_{k} \\
\lambda_{k} = \lambda_{k - 1} + \gamma \Delta \lambda_{k}
\end{cases}
\end{equation*}
with a suitably chosen step-length $\gamma$. This gives us a corresponding update for $x_{k - 1}$:
\begin{align*}
x_{k} = V_{k} y_{k} &= V_{k}\bar{y}_{k-1} + \gamma  V_{k}\Delta y_{k} \\
 &= V_{k - 1}y_{k - 1} + \gamma  V_{k}\Delta y_{k} \\
 &= x_{k - 1} + \gamma  \underbrace{V_{k}\Delta y_{k}}_{:= \Delta x_{k}}.
\end{align*}
By multiplying the Newton step $\Delta y_{k}$ for the projected variable $\bar{y}_{k-1}$ by $V_k$ we obtain a step $\Delta x_{k}$ for $x_{k-1}$. Note that this step is different from the Newton step that would be obtained by solving \cref{eq:newton_system} for $(x_{k-1},\lambda_{k-1})$. However, we will show that the step $(\Delta x_k^T,\Delta \lambda_k)^T$ is a descent direction for $f(x_{k-1},\lambda_{k-1})$, which is the main result of the current section. 
We start by proving the following lemma: 
\begin{lemma} \label{thm:lemma1}
Let $F(x,\lambda)$ be defined as \cref{eq:F} and $F^{(k)}(y,\lambda)$  as \cref{eq:projected_f}. Furthermore let $\bar{y}_{k - 1}=(y_{k - 1}^T,0)^T  \in\mathbb{R}^k $  and $x_{k-1}\in\mathbb{R}^n$ be such that $x_{k-1} = V_{k}\bar{y}_{k-1}$ for the orthonormal basis $V_k$ generated by Bidiag1. Then we have following equality: 
\begin{equation} \label{eq:f_equal_proj}
||F(x_{k-1},\lambda_{k-1})|| = ||F^{(k)}(\bar{y}_{k-1},\lambda_{k-1})||. 
\end{equation}
\end{lemma}

\begin{proof}
First note that we can write
\begin{equation*}
||F(x_{k-1},\lambda_{k-1})||^2 = ||\lambda_{k-1} A^T(Ax_{k-1} - b) + x_{k-1}||^2 +  \left( \frac{1}{2}||Ax_{k-1} - b||^2 - \frac{\sigma^2}{2}\right)^2.
\end{equation*}
Similarly as shown in \cref{eq:norm_equal} we have that $||Ax_{k-1} - b|| = ||B_{k+1,k}\bar{y}_{k-1} - c_{k+1}||$. Now let us take a closer look at the first term:
 \begin{eqnarray*}
&&||\lambda_{k-1} A^T(Ax_{k-1} - b) + x_{k-1}|| \\
&=&||\lambda_{k-1} A^T(AV_{k}\bar{y}_{k-1} - b) + V_{k}\bar{y}_{k-1}||\\
&\stackrel{\cref{eq:avub}}{=} &||\lambda_{k-1} A^TU_{k+1}(B_{k+1,k}\bar{y}_{k-1} - c_{k+1}) + V_{k}\bar{y}_{k-1}|| \\
&\stackrel{\cref{eq:atu}}{=} &||\lambda_{k-1} (V_{k}B_{k+1,k}^T + \mu_{k}v_{k} e_{k + 1}^T)(B_{k+1,k}\bar{y}_{k-1} - c_{k+1}) + V_{k}\bar{y}_{k-1}|| \\
&=&  ||\lambda_{k-1} V_{k}B_{k+1,k}^T(B_{k+1,k}\bar{y}_{k-1} - c_{k+1}) + V_{k}\bar{y}_{k-1}|| \\
&=&||\lambda_{k-1} B_{k+1,k}^T(B_{k+1,k}\bar{y}_{k-1} - c_{k+1}) + \bar{y}_{k-1}||.
\end{eqnarray*}
The second to last equality follows from the fact that the last element of $(B_{k+1,k}\bar{y}_{k-1} - c_{k+1})$ is zero and that the matrices $V_{k}B_{k+1,k}^T$ and
$V_{k}B_{k+1,k}^T + \mu_{k}v_{k} e_{k + 1}^T$ only differ in the last column. Now the proof follows from the fact that we have
\begin{multline*}
||F^{(k)}(\bar{y}_{k-1},\lambda_{k-1})||^2 =||\lambda_{k-1} B_{k+1,k}^T(B_{k+1,k}\bar{y}_{k-1} - c_{k+1}) + \bar{y}_{k-1}||^2 \\
 + \left( \frac{1}{2}||B_{k+1,k}\bar{y}_{k-1} - c_{k+1}||^2 -  \frac{\sigma^2}{2}\right)^2.
 \end{multline*}

\end{proof}

\begin{lemma} \label{thm:lemma2}
Let $F(x,\lambda)$ be defined as \cref{eq:F} and $F^{(k)}(y,\lambda)$  as \cref{eq:projected_f}. Furthermore let $\bar{y}_{k - 1}=(y_{k - 1}^T,0)^T  \in\mathbb{R}^k $  and $x_{k-1}\in\mathbb{R}^n$ be such that $x_{k-1} = V_{k}\bar{y}_{k-1}$ for the orthonormal basis $V_k$ generated by Bidiag1. Then we have following equality: 
\begin{equation} \label{eq:equal_jac}
\begin{pmatrix} V_{k}^T & 0 \\ 0 & 1 \end{pmatrix} J(x_{k-1},\lambda_{k-1})F(x_{k-1},\lambda_{k-1}) = J^{(k)}(\bar{y}_{k-1},\lambda_{k-1}) F^{(k)}(\bar{y}_{k-1},\lambda_{k-1}).
\end{equation}
\end{lemma}

\begin{proof}
Let us first introduce notations $t_{k+1} = B_{k + 1,k} \bar{y}_{k-1} - c_{k + 1}$ and
\begin{equation*}
\zeta_k = \frac{1}{2} \left(||Ax_{k-1}-b||^2 -  \sigma^2 \right) = \frac{1}{2}\left(||t_{k+1}||^2 - \sigma^2 \right).
\end{equation*}
 The left-hand side of equality \cref{eq:equal_jac} can be rewritten as
\begin{align*}
&\begin{pmatrix} V_{k}^T & 0 \\ 0 & 1 \end{pmatrix} J(x_{k-1},\lambda_{k-1})F(x_{k-1},\lambda_{k-1}) \\
=&\begin{pmatrix} V_{k}^T & 0 \\ 0 & 1 \end{pmatrix} \begin{pmatrix}\lambda_{k-1} A^T A + I  &  A^T(Ax_{k-1} - b) \\ (Ax_{k-1}-b)^T A & 0 \end{pmatrix}\begin{pmatrix} \lambda_{k-1} A^T(Ax_{k-1} - b) + x_{k-1} \\ \zeta_{k}\end{pmatrix} \\
=& \begin{pmatrix}\lambda_{k-1} V_{k}^TA^T A + V_{k}^T & V_{k}^T A^T(Ax_{k-1} - b) \\ (Ax_{k-1}-b)^T A & 0 \end{pmatrix}\begin{pmatrix} \lambda_{k-1} A^T(Ax_{k-1} - b) + x_{k-1} \\ \zeta_{k} \end{pmatrix} \\
=& \begin{pmatrix} (\lambda_{k-1} V_{k}^TA^T A + V_{k}^T )(\lambda_{k-1} A^T(Ax_{k-1} - b) + x_{k-1}) + \zeta_{k}V_{k}^T A^T(Ax_{k-1} - b) \\  \lambda_{k-1} (Ax_{k-1}-b)^T A A^T(Ax_{k-1}- b) + (Ax_{k-1}-b)^T A x_{k-1}\end{pmatrix}.
\end{align*}
Similarly we can rewrite the right-hand side of equality \cref{eq:equal_jac} side as
\begin{eqnarray*}
&&J^{(k)}(\bar{y}_{k-1},\lambda_{k-1}) F^{(k)}(\bar{y}_{k-1},\lambda_{k-1}) \\
&=& \begin{pmatrix} \lambda_{k-1} B_{k + 1,k}^T B_{k + 1,k} + I_{k} &  B_{k + 1,k}^T t_{k+1}\\  t_{k+1}^T B_{k + 1,k} & 0 \end{pmatrix}\begin{pmatrix} \lambda_{k-1} B_{k + 1,k}^T t_{k+1} +\bar{y}_{k-1} \\ \zeta_{k}\end{pmatrix} \\
&=& \begin{pmatrix} (\lambda_{k-1} B_{k + 1,k}^T B_{k + 1,k} + I_{k} )(\lambda_{k-1} B_{k + 1,k}^T t_{k+1} + \bar{y}_{k-1}) + \zeta_{k} B_{k + 1,k}^T t_{k+1}\\ \lambda_{k-1}t_{k+1} ^T B_{k + 1,k}B_{k + 1,k}^T t_{k+1} +t_{k+1}^T B_{k + 1,k}\bar{y}_{k-1}\end{pmatrix}.
\end{eqnarray*}
We can now compare all individual terms and check if they are indeed equal.
Let us work out the second block as an example: 
\begin{eqnarray*}
&&\lambda_{k-1} (Ax_{k-1}-b)^T A A^T(Ax_{k-1}- b) + (Ax_{k-1}-b)^T A x_{k-1} \\
&=& \lambda_{k-1} (AV_{k}\bar{y}_{k-1}-b)^T A A^T(AV_{k}\bar{y}_{k-1}- b) + (AV_{k}\bar{y}_{k-1}-b)^T AV_{k}\bar{y}_{k-1} \\
&=& \lambda_{k-1} t_{k+1}^T U_{k+1}^T A A^T U_{k+1}t_{k+1} + t_{k+1}^T U_{k+1}^T AV_{k}\bar{y}_{k-1} \\
&=& \lambda_{k-1} t_{k+1}^T U_{k+1}^T A A^T U_{k+1}t_{k+1} + t_{k+1}^T B_{k+1,k}\bar{y}_{k-1} \\
&=& \lambda_{k-1} t_{k+1}^T B_{k+1,k} B_{k+1,k}^T t_{k+1} + t_{k+1}^T B_{k+1,k}\bar{y}_{k-1}. 
\end{eqnarray*}
The last equality follows from the fact that $U_{k+1} A A^T U_{k+1}$ and $B_{k+1,k} B_{k+1,k}^T$ only differ in the last column and that the last element of $t_{k+1}$ is equal to zero. Indeed, using \cref{eq:atu} we have
\begin{eqnarray*}
&& U_{k+1}^T A A^T U_{k+1} \\
&=& \left(V_{k}B_{k+1,k}^T + \mu_{k}v_{k} e_{k + 1}^T  \right)^T\left( V_{k}B_{k+1,k}^T + \mu_{k}v_{k} e_{k + 1}^T\right) \\
&=& B_{k+1,k} \underbrace{V_{k}^T V_{k}}_{= I_{k}} B_{k+1,k}^T + 2\mu_{k} \underbrace{e_{k+1} v_{k}^T V_{k}}_{=0_{k+1,k}}B_{k+1,k}^T +\mu^2_{k}e_{k+1} \underbrace{v_{k}^T v_{k}}_{=1} e_{k+1}^T  \\
&=&B_{k+1,k} B_{k+1,k}^T +\mu^2_{k}e_{k+1}e_{k+1}^T, 
\end{eqnarray*}
where $0_{k+1,k}$ is a $(k+1)\times k$ matrix containing all zeros. This shows that these matrices only differ in the last element. Equality of the first component can be proven similarly. 
\end{proof}

\Cref{thm:lemma1,thm:lemma2} allow us to prove the main result of the current section:

\begin{theorem}\label{thm:descent}
Let $\begin{pmatrix}\Delta y_k \\\Delta\lambda_k \end{pmatrix}$ be defined as in \cref{eq:newtoneq} and let $\Delta x_{k} = V_{k}\Delta y_k$, where $V_{k}$ is the orthonormal matrix generated by Bidiag1. Then we have
\begin{equation}
\begin{pmatrix} \Delta x_{k} \\ \Delta \lambda_{k} \end{pmatrix}^T \nabla f(x_{k - 1},\lambda_{k -1}) = -||F(x_{k-1},\lambda_{k-1})||^2 \leq 0
\end{equation}
which means that $\begin{pmatrix} \Delta x_{k} \\ \Delta \lambda_{k} \end{pmatrix}$ is a descent direction for $f(x_{k-1},\lambda_{k-1})$.
\end{theorem}

\begin{proof}

The result now follows from an easy calculation:  
\begin{align*}
& \begin{pmatrix} \Delta x_{k} \\ \Delta \lambda_{k} \end{pmatrix}^T \nabla f(x_{k - 1},\lambda_{k -1}) \\
=& \begin{pmatrix} V_k \Delta y_{k} \\ \Delta \lambda_{k} \end{pmatrix}^T J(x_{k-1},\lambda_{k-1})F(x_{k-1},\lambda_{k-1}) \\
=& \begin{pmatrix} \Delta y_{k} \\ \Delta \lambda_{k} \end{pmatrix}^T \begin{pmatrix} V_{k}^T & 0 \\ 0 & 1 \end{pmatrix} J(x_{k-1},\lambda_{k-1})F(x_{k-1},\lambda_{k-1}) \\
\stackrel{\cref{eq:equal_jac}}{=}& \begin{pmatrix} \Delta y_{k} \\ \Delta \lambda_{k} \end{pmatrix}^T J^{(k)}(\bar{y}_{k-1},\lambda_{k-1}) F^{(k)}(\bar{y}_{k-1},\lambda_{k}) \\
\stackrel{\cref{eq:newtoneq}}{=}& -  \left( J^{(k)}(\bar{y}_{k-1},\lambda_{k-1}) ^{-1} F^{(k)}(\bar{y}_{k-1},\lambda_{k-1}) \right)^T J^{(k)}(\bar{y}_{k-1},\lambda_{k-1}) F^{(k)}(\bar{y}_{k-1},\lambda_{k-1}) \\
=& - ||F^{(k)}(\bar{y}_{k-1},\lambda_{k-1}) ||^2  \\
\stackrel{\cref{eq:f_equal_proj}}{=}& - ||F(x_{k-1},\lambda_{k-1})||^2 \leq 0. 
\end{align*}
\end{proof}

\Cref{thm:descent} implies that we can compute the Newton direction for the projected function $F^{(k)}(\bar{y}_{k-1},\lambda_{k-1})$, which we refer to as a \textit{projected Newton direction}, and obtain a descent direction for $f(x_{k-1},\lambda_{k-1})$ by multiplying the step $\Delta y_k$ with the Krylov subspace basis $V_k$. Alternatively, this can be seen as performing a single Newton iteration on the projected minimization problem of dimension $k$ with initial guess $(y_{k-1}^T,0)^T$, where $y_{k-1}$ is the approximate solution of the $k-1$ dimensional projected optimization problem and then multiplying the result $y_k$ with $V_{k}$ to obtain an approximate solution $x_{k} = V_k y_k $ of the noise constrained Tikhonov problem \cref{eq:noise_constrained}. Subsequently, the Krylov subspace basis is expanded and the procedure is repeated until a sufficiently accurate solution is found. We will use $||F(x_{k},\lambda_{k})||$ to monitor convergence and stop the iterations when this value is smaller than some prescribed tolerance $tol$. This norm can be calculated very efficiently using the Krylov subspace, see \cref{thm:remark_fbar}. Since in practice the quality of the solution will not improve drastically when requiring a very accurate solution, we suggest using a modest tolerance such as $tol = 10^{-3}$. 

When $k \ll n$ calculating the projected Newton step is much cheaper than calculating the true Newton direction for $F(x_{k-1},\lambda_{k-1})$, since the former requires us to solve a $(k + 1)\times (k + 1)$ linear system, while the latter is found by solving an $(n + 1) \times (n + 1)$ linear system. Moreover, using \cref{thm:descent} we can show the existence of a step-length $\gamma_k > 0$ such that 
\begin{equation} \label{eq:sufficient_decrease}
\frac{1}{2} ||F(x_{k},\lambda_{k})||^2 \leq \left(\frac{1}{2} - c\gamma_k \right) ||F(x_{k - 1},\lambda_{k - 1})||^2
\end{equation}
with $c\in (0,1)$, provided that $||F(x_{k - 1},\lambda_{k - 1})||\neq 0$. Equation \cref{eq:sufficient_decrease} is often referred to in literature as the sufficient decrease condition or Armijo condition. The following theorem can be used to show existence of a suitable step-length:

\begin{theorem}[\cite{blowey2012frontiers}] \label{thm:termination}
Suppose that $f:\mathbb{R}^{n+1}\rightarrow \mathbb{R}$ is continuously differentiable and $\nabla f$ is Lipschitz continuous in $u\in\mathbb{R}^{n}$ with Lipschitz constant $\zeta(u)$, that $c\in(0,1)$ and that $p$ is a descent direction at $u$, i.e $\nabla f(u)^T p < 0$. Then the Armijo condition 
\begin{equation*}
f(u + \gamma p) \leq f(u) + c\gamma \nabla f(u)^T p
\end{equation*}
is satisfied for all $\gamma \in [0,\gamma_{\text{max}}]$ where
\begin{equation*}
\gamma_{\text{max}} = \frac{2(c - 1)\nabla f(u)^T p}{\zeta(u) ||p||^2}.
\end{equation*}
\end{theorem} 
A step-length $\gamma_k$ for which the sufficient decrease condition \cref{eq:sufficient_decrease} holds is can be found using a so-called backtracking line search. We start by setting $\gamma_k = 1$ (which corresponds to taking a full Newton step) and check if the condition holds. If not, we decrease $\gamma$ by a factor $\tau \in (0,1)$, say $\tau = 0.9$, and check if $\gamma_k := \tau \gamma_k$ satisfies the condition. This procedure is then repeated until a suitable step-length is found. Using \cref{thm:termination} it is possible to show that this procedure eventually terminates:
\begin{lemma} \label{thm:finter}
Let $\begin{pmatrix}\Delta y_k \\\Delta\lambda_k \end{pmatrix}$ be defined as in \cref{eq:newtoneq} and let $\Delta x_{k} = V_{k}\Delta y_k$, where $V_{k}$ is the orthonormal matrix generated by Bidiag1 and denote $p_k =  \begin{pmatrix}\Delta x_k \\\Delta\lambda_k \end{pmatrix}$. Suppose $\nabla f(x,\lambda)$ is Lipschitz continuous on the level set $L = \{(x,\lambda) : ||F(x,\lambda)||  \leq ||F(x_0,\lambda_0)|| \}$ with constant $\zeta$. Then for all $k\geq1$ the backtracking line search procedure terminates with
\begin{equation} \label{eq:lower}
\gamma_k \geq \min \left(1,\frac{2(1 - c)\tau ||F(x_{k - 1},\lambda_{k-1})||^2}{ \zeta||p_k||^2}\right).
\end{equation}
 
\end{lemma}
\begin{proof}
This easily follows from combining \cref{thm:descent,thm:termination}.
\end{proof}

\begin{remark} \label{thm:remark_fbar}
The sufficient decrease condition \cref{eq:sufficient_decrease} can be efficiently computed in the Krylov subspace. First of all, from \cref{thm:lemma1} we know how to compute the previous residual norm  using $F^{(k)}(y,\lambda)$. So what we need to know is how to compute $||F(x_{k},y_{k})||$. Let us denote $B_{k,k} \in\mathbb{R}^{k  \times k}$ for $k\geq 1 $ the square matrix containing the first $k$ columns of $B_{k + 1,k}$. We can then show similarly as in the proof of \cref{thm:lemma1} that $||F(x_{k},y_{k})|| = ||\bar{F}^{(k)}(\bar{y}_{k},\lambda_{k})||$ with
\begin{equation*}
\bar{F}^{(k )}(y,\lambda) =  \begin{pmatrix}  \lambda B_{k + 1,k+1}^T (B_{k + 1,k + 1} y - c_{k + 1}) + y \\ \frac{1}{2} ||B_{k + 1,k + 1} y - c_{k + 1}||^2 - \frac{1}{2} \epsilon^2\end{pmatrix}
\end{equation*}
Note that the only difference between $\bar{F}^{(k)}$ and $F^{(k)}$ is the multiplication with the matrix $B_{k+1,k+1}$ instead of $B_{k+1,k}$. This norm can thus be computed if we have the basis $V_{k + 1}, U_{k + 1}$ and matrix $B_{k + 1, k + 1}$ in iteration $k$. 
\end{remark}

Following the discussion above, we can now formulate the Projected Newton method, see \cref{alg:PNTM}.

\begin{algorithm}
\caption{Projected Newton method \hfill \textbf{Input:} $A$,$b$,$\lambda_0, \sigma$, tol, maxit}
\label{alg:PNTM}
\begin{algorithmic}[1]
\STATE{ $\bar{y}_0 = 0$; $x_{0} = 0$; $\tau = 0.9$; $c = 10^{-4}$; \hfill $\#$ Initialization }
\STATE{ $u_{0} = b/||b||$; $r_{0} = A^T u_{0};$ $\mu_0 = ||r_{0}||;$ $v_0 = r_0/\mu_0;$  \hfill  $\#$ Calculate $B_{1,1}, V_{1}$ and $U_{1}$}
\FOR{$k = 1,2,\ldots,\text{maxit}$} 
\STATE{ $p_{k - 1} = Av_{k - 1} - \mu_{k - 1} u_{k - 1};$}
\STATE{ Reorthogonalize $p_ {k - 1}$ w.r.t. $U_{k}$ if necessary.}
\STATE{ $\nu_{k} = ||p_{k-1}||;$ $ u_{k} = p_{k - 1}/\nu_{k};$  \hfill  $\#$ Calculate $B_{k+1,k}$ and $U_{k + 1}$}
\STATE{ $ r_{k} = A^T u_{k} - \nu_{k} v_{k - 1};$}
\STATE{ Reorthogonalize $r_k$ w.r.t. $V_k$ if necessary.}
\STATE{ $ \mu_{k} = ||r_{k}||;$ $v_{k} = r_{k}/\mu_{k};$  \hfill $\#$ Calculate $B_{k+1,k+1}$ and $V_{k + 1}$}
\STATE{ $\begin{pmatrix} \Delta y_{k}\\ \Delta \lambda_{k} \end{pmatrix} = -  J^{(k)}(\bar{y}_{k-1},\lambda_{k-1}) ^{-1} F^{(k)}(\bar{y}_{k-1},\lambda_{k-1});$  \hfill $\#$ Calculate Newton step}
\STATE{$\lambda_{k} = \lambda_{k-1} + \Delta \lambda_{k};$}   
\IF{$\lambda_{k}> 0$} 
\STATE{$\gamma_k = 1;$}
\ELSE
\STATE{$\gamma_k = -\tau (\lambda_{k-1} / \Delta \lambda_{k}) ; \lambda_{k} = \lambda_{k-1} + \gamma_k \Delta \lambda_{k};$ \hfill $\#$ Ensure positivity $\lambda_{k}$}
\ENDIF
\STATE{ $y_{k} = \bar{y}_{k-1} + \gamma_k \Delta y_{k};$ $\bar{y}_{k} = (y_{k}^T,0)^T;$}
\WHILE {$\frac{1}{2}||\bar{F}^{(k)}(\bar{y}_{k},\lambda_{k})||^2 \geq (\frac{1}{2} - c\gamma_k) || F^{(k)}(\bar{y}_{k-1},\lambda_{k-1}) || ^2$  \label{lineback}}
\STATE{ $\gamma_k = \tau \gamma_{k}$; $\lambda_{k} = \lambda_{k-1} + \gamma_k \Delta \lambda_{k};$  \hfill $\#$ Backtracking line search}
\STATE{ $y_{k} =\bar{y}_{k-1} + \gamma_k \Delta y_{k};$ $\bar{y}_{k} = (y_{k}^T,0)^T;$ }
\ENDWHILE

\IF{$\left( ||{\bar{F}}^{(k)}(\bar{y}_{k},\lambda_{k})|| \leq \text{tol} \right) $} 
\STATE{ \textbf{RETURN}  $x_k = V_ky_k,\lambda_k$}
\ENDIF
\ENDFOR
\end{algorithmic}
\end{algorithm}

\subsection{General form Tikhonov regularization}
\Cref{alg:PNTM} can be adapted to allow for general form Tikhonov regularization. In its general form, the Tikhonov problem is written as
\begin{equation}\label{eq:gentikhonov}
		x_\alpha = \argmin_{x\in\mathbb{R}^n}\frac{1}{2} \left\|Ax - b\right\|^2 + \frac{\alpha}{2}\left\|L(x - x_0)\right\|^2,
\end{equation}
with $x_0\in\mbbR^n$ an initial estimate and $L\in\mbbR^{p\times n}$ a regularization
matrix, both chosen to incorporate prior knowledge or to place specific constraints
on the solution \cite{gazzola2014_1, hansen2010}. If $L$ is a square invertible matrix, then the problem can
be written in the standard form 
\begin{equation}\label{eq:tranftikhonov}
			z_\alpha = \argmin_{z\in\mathbb{R}^n}\frac{1}{2}\left\|\bar{A}z - r_0\right\|^2 + \frac{\alpha}{2}\left\|z\right\|^2,
\end{equation}
by using the transformation
\begin{equation*}
		z = L(x - x_0), \quad \bar{A} = AL^{-1}, \quad r_0 = b - Ax_0.
\end{equation*}
When $L$ is not square invertible, some form of pseudoinverse has to be used,
but the reformulation of the problem remains the same \cite{hansen2010}.

After solving \cref{eq:tranftikhonov}, the solution of \cref{eq:gentikhonov} can be found as
\[
		x = x_0 + L^{-1}z.
\]
It is important to note that this transformation is only of practical interest if the matrix $L$ is cheaply invertible, or if the linear system $Ly =z$ can be solved efficiently. For large-scale problems, this is of course not the case for all choices of regularization matrix. When the matrix $L$ is banded an alternative way to transform \cref{eq:gentikhonov} into standard form is suggested in \cite{elden1977algorithms}. The author suggests to transform the problem by using a QR decomposition of $L^T$, which he argues can be calculated very efficiently using a sequence of plane rotations. We refer to his work for more details on this topic \cite{elden1977algorithms}. 

Unfortunately, \cref{thm:lemma1} and \cref{thm:lemma2} fail to generalize using Bidiag1 if we explicitly use the regularization term $\left\|L(x - x_0)\right\|^2$ and as a consequence the projected Newton direction does not result in a descent direction, see also \cref{thm:descent}. One idea is to use the extension of the Golub-Kahan bidiagonalization algorithm given in \cite{hochstenbach2015golub} that constructs a decomposition of a matrix pair $\{A,L\}$. This extension will be explored in future research.

\begin{remark} \label{remark_general}
In its current form, the Projected Newton method can only be used for (general form) Tikhonov regularization. However, in many applications a different regularization functional $\phi(x)$ produces better reconstructions. For instance, when it is important to preserve edges in a reconstructed image, the total variation functional $TV(x)$ is a good candidate \cite{rudin1992nonlinear,strong2003edge,tian2011low} . Another popular regularization functional is the $\ell_1$-norm $||\cdot||_1$, which is known to produce a sparse solution and is often referred to in literature as Basis Pursuit denoising \cite{chen2001atomic,wright2009sparse,gill2011crowd}. Hence, we would like to be able to solve the more general regularization problem \cref{eq:general_reg}. However, to reformulate the projected minimization problem \cref{projected_minequations1} to the more convenient form \cref{projected_minequations2} we use the fact that $\phi(x_k) = \phi(V_k y_k) = \phi(y_k)$ for the Tikhonov functional $\phi(x) = ||x||^2/2$, which is not true in general. Hence, at the current time, it is unclear how the projection step can be generalized for other regularization terms. 

One idea is to solve a series of Tikhonov problems using the Projected Newton method, where we approximate the regularization term $\phi(x)$ with a regularization term of the form $||L x ||^2_2$ and then improve the approximation of the regularization term based on the obtained solution. Every subsequent Tikhonov problem would then be a better approximation of the general regularization problem \cref{eq:general_reg}. Similar approaches have been taken in \cite{gazzola2014_1,rodriguez2006iteratively,rodriguez2008efficient}.
\end{remark}

\section{Proof of convergence} \label{sec:proofconv}

In this section we prove a convergence result for the iterates generated in \cref{alg:PNTM}. First we show that the algorithm does not break down, i.e that the Jacobian matrix is never singular.
\begin{lemma}
The Jacobian matrix $J^{(k)}(\bar{y}_{k - 1},\lambda_{k - 1})$ constructed in iteration $k$ of \cref{alg:PNTM} is nonsingular. 
\end{lemma}
\begin{proof}
Since we enforce positivity of the regularization parameter $\lambda_{k-1}$ the only thing we need to show is that 
\begin{equation}
B_{k + 1,k}^T(B_{k + 1,k} \bar{y}_{k-1} - c_{k+1}) \neq 0.
\end{equation}
Using the fact that the last component of $\bar{y}_{k-1}$ is zero, we can write:
\begin{equation*}
 ||B_{k + 1,k}^T(B_{k + 1,k} \bar{y}_{k-1} - c_{k+1})|| = ||B_{k + 1,k+1}^T(B_{k + 1,k} \bar{y}_{k-1} - c_{k+1})||. 
\end{equation*}
Recall that $B_{k + 1,k+1}^T$ has a trivial nullspace since we assume the bidiagonalization procedure does not break down, see \cref{thm:breakdown}. Hence, what remains to be shown is that $||B_{k + 1,k} \bar{y}_{k-1} - c_{k+1}||>0$ for all $k\geq 1$. In fact, we will prove by induction that $||B_{k + 1,k} \bar{y}_{k-1} - c_{k+1}||\geq\sigma$ for all $k$. For $k=1$ we have 
\begin{equation}
||B_{2,1} \bar{y}_{0} - c_{2}|| = ||c_{2}|| = ||b||\geq\sigma.
\end{equation}
Let us now consider $k>1$ and assume  $||B_{k,k-1}\bar{y}_{k-2} - c_{k}|| \geq \sigma$ which implies the previous Jacobian $J^{(k - 1)}(\bar{y}_{k - 2},\lambda_{k - 2})$ was nonsingular. This allows us to solve the linear system
\begin{equation*}
\begin{pmatrix} \lambda_{k-1} B_{k,k-1}^T B_{k,k-1} + I_{k-1} &  B_{k,k-1}^T( B_{k,k-1}\bar{y}_{k-2} - c_{k}) \\  ( B_{k,k-1} \bar{y}_{k-2}  - c_{k})^T B_{k,k-1} & 0 \end{pmatrix}\begin{pmatrix} \Delta y_{k - 1} \\\Delta \lambda_{k-1} \end{pmatrix} = -F^{(k-1)}(\bar{y}_{k-2},\lambda_{k-2}).
\end{equation*}
Writing out the last component of this equality, we get
\begin{equation} \label{eq:last_comp}
 ( B_{k,k-1}\bar{y}_{k-2} - c_{k})^T B_{k,k-1}\Delta y_{k-1} = - \frac{1}{2}( ||B_{k,k-1}\bar{y}_{k-2} - c_{k}||^2 - \sigma^2) \leq 0
\end{equation}
where the inequality follows from the induction hypothesis. Now let $y_{k-1} = \bar{y}_{k-2} + \gamma_{k-1}\Delta y_{k-1}$ then we have
\begin{align*}
&||B_{k,k-1}y_{k-1}- c_{k}||^2 = ||B_{k,k-1}\bar{y}_{k-2} - c_{k} + \gamma_{k-1}B_{k,k-1}\Delta y_{k-1}||^2 \\
&= ||B_{k,k-1}\bar{y}_{k-2} - c_{k}||^2 + 2\gamma_{k-1}(B_{k,k-1}\bar{y}_{k-2} - c_{k})^T B_{k,k-1}\Delta y_{k-1} + \gamma_{k-1}^2||B_{k,k-1}\Delta y_{k-1}||^2 \\
&\geq ||B_{k,k-1}\bar{y}_{k-2} - c_{k}||^2 + 2(B_{k,k-1}\bar{y}_{k-2} - c_{k})^T B_{k,k-1}\Delta y_{k-1} + \gamma_{k-1}^2||B_{k,k-1}\Delta y_{k-1}||^2 \\
&\stackrel{\cref{eq:last_comp}}{=} \sigma^2  + \gamma_{k-1}^2||B_{k,k-1}\Delta y_{k-1}||^2.
\end{align*}
So we know that $||B_{k + 1,k} \bar{y}_{k-1} - c_{k+1}||=||B_{k,k-1}y_{k-1}- c_{k}|| \geq \sigma$ which concludes the proof.
\end{proof}

We can now prove the following convergence theorem, which can be seen as a corollary of Theorem $2.2$ in \cite{blowey2012frontiers}:
\begin{theorem}\label{thm:convthm}
Let $\{(x_{k},\lambda_{k})\}_{k\in\mathbb{N}}$ be the iterates generated by \cref{alg:PNTM}. Suppose the lower bound \cref{eq:lower} in \cref{thm:finter} holds for all $k$, then one of the following statements is true:
\begin{eqnarray*}
&& \bullet \hspace{0.5cm} || F(x_k,\lambda_k) || = 0 \hspace{0.5cm}\text{for some } \hspace{0.5cm} k\geq0 \\
&&  \bullet \hspace{0.5cm} \lim_{k\rightarrow \infty} \min \left(||F(x_k,\lambda_k)||,\frac{||F(x_k,\lambda_k)||^2}{||p_{k+1}||}\right)=0. 
\end{eqnarray*}
\end{theorem}
\begin{proof}
The proof of this theorem is a direct modification of the proof of Theorem $2.2$ in \cite{blowey2012frontiers}.
Suppose $|| F(x_k,\lambda_k) ||\neq 0$ for all $k$, since otherwise we are done. Then the search directions produce a descent direction as shown in \cref{thm:descent}. From the backtracking line search it follows that
\begin{equation*}
||F(x_{j+1},\lambda_{j+1})||^2 - ||F(x_{j},\lambda_{j})||^2 \leq  -2c \gamma_{j+1} ||F(x_{j},\lambda_{j})||^2
\end{equation*}
for all $j\leq k$. Taking the sum over all previous iterations gives us the following inequality
\begin{equation*}
||F(x_{k+1},\lambda_{k+1})||^2 - ||F(x_{0},\lambda_{0})||^2 \leq  -2c \sum_{j=0}^{k}\gamma_{j+1} ||F(x_{j},\lambda_{j})||^2.
\end{equation*}
By rearranging the terms, taking the limit and using the fact that $||F(x_{k+1},\lambda_{k+1})||\geq 0 $ for all $k$ we get
\begin{equation*}
 2c \sum_{j=0}^{\infty}\gamma_{j+1}||F(x_{j},\lambda_{j})||^2 \leq \lim_{k\rightarrow \infty} \left[ ||F(x_{0},\lambda_{0})||^2 - ||F(x_{k+1},\lambda_{k+1})||^2 \right] < \infty
\end{equation*}
which shows we have a bounded series. Moreover, all terms in the series are positive, which shows $\sum_{j=0}^{\infty}\gamma_{j+1} ||F(x_{j},\lambda_{j})||^2$ is convergent, from which it follows that
\begin{equation}\label{eq:lim}
 \lim_{k\rightarrow \infty} \gamma_{k+1}|| F(x_k,\lambda_k) ||^2 = 0.
\end{equation}
Let us consider a partition of $\mathbb{N}=\mathcal{N}_1 \cup \mathcal{N}_2$ with
\begin{equation*}
\mathcal{N}_1 = \{k : \gamma_k = 1\} \hspace{0.5cm}\text{and} \hspace{0.5cm} \mathcal{N}_2 = \{k : \gamma_k < 1\}.
\end{equation*}
From \cref{eq:lim} it now follows for the first index set that
\begin{equation*}
\lim_{k\in\mathcal{N}_1\rightarrow \infty} || F(x_k,\lambda_k) ||^2 = \lim_{k\in\mathcal{N}_1\rightarrow \infty} \gamma_{k+1}|| F(x_k,\lambda_k) ||^2 = 0.
\end{equation*}
So let us take a closer look at the second index set for which we also have
\begin{equation*}
\lim_{k\in\mathcal{N}_2\rightarrow \infty} \gamma_{k+1}|| F(x_k,\lambda_k) ||^2 = 0.
\end{equation*}
 For $k\in\mathcal{N}_2$ we have using \cref{eq:lower} that 
\begin{equation*}
0\leq \frac{2(1 - c)\tau ||F(x_{k},\lambda_{k})||^4}{ \zeta||p_{k+1}||^2} \leq \gamma_{k+1}|| F(x_k,\lambda_k) ||^2
\end{equation*}
from which we get
\begin{equation*}
\lim_{k\in\mathcal{N}_2\rightarrow \infty} \frac{||F(x_k,\lambda_k)||^2}{||p_{k+1}||}=0.
\end{equation*}
The proof now follows from combining the result for both index sets. 
\end{proof}
\begin{remark}
The limit in \cref{thm:convthm} can be simplified if $||J^{(k)}(\bar{y}_{k-1},\lambda_{k-1})^{-1}|| \leq M$ for some constant independent of $k$. Indeed, since 
\begin{eqnarray*}
\| p_{k} \| &=&\| \begin{pmatrix} \Delta x_{k} \\ \Delta \lambda_{k} \end{pmatrix} \| = \| \begin{pmatrix} \Delta y_{k} \\ \Delta \lambda_{k} \end{pmatrix} \| \\ &=& || J^{(k)}(\bar{y}_{k-1},\lambda_{k-1}) ^{-1} F^{(k)}(\bar{y}_{k-1},\lambda_{k-1})|| \\
&\leq&  M|| F^{(k)}(\bar{y}_{k-1},\lambda_{k-1})|| = M||F(x_{k-1},\lambda_{k-1})|| 
\end{eqnarray*}
This allows us to write
\begin{equation}
\frac{||F(x_{k-1},\lambda_{k-1})||^2}{||p_{k}||}\geq \frac{||F(x_{k-1},\lambda_{k-1})||^2}{M||F(x_{k-1},x_{k-1})||} = \frac{1}{M}||F(x_{k-1},\lambda_{k-1})||\geq 0
\end{equation}
from which we can conclude that
\begin{equation}
\lim_{k\rightarrow \infty} \min \left(||F(x_k,\lambda_k)||,\frac{||F(x_k,\lambda_k)||^2}{||p_{k+1}||}\right) =\lim_{k\rightarrow \infty} ||F(x_k,\lambda_k)|| = 0. \end{equation}
\end{remark}
\section{Numerical experiments} \label{sec:numex}
In this section we report the results of some numerical experiments with test problems from image deblurring and computed tomography. Moreover, to thoroughly test the robustness of the Projected Newton method, we apply the algorithm to 164 matrices from the SuiteSparse matrix collection \cite{davis2011}. We start this section by explaining another Krylov subspace method for automatic regularization based on the discrepancy principle, namely Generalized bidiagonal-Tikhonov, which we use to compare our newly developed method. 
\subsection{Reference method: Generalized bidiagonal-Tikhonov}
In \cite{gazzola2014_1, gazzola2014_2, gazzola2015krylov} a generalized Arnoldi-Tikhonov
method (GAT) was introduced that iteratively solves the Tikhonov problem \eqref{eq:min_tikhonov}
using a Krylov subspace method based on the Arnoldi decomposition of the matrix $A$.
Simultaneously, after each Krylov iteration, the regularization parameter is updated
in order to approximate the value for which the discrepancy is equal to $\eta\varepsilon$.
This is done using one step of the secant method to find the intersection of the
discrepancy curve with the tolerance for the discrepancy principle, see \cref{fig:curves}, but in the current Krylov subspace. Because the method is based on
the Arnoldi decomposition, it only works for square matrices. However, by replacing the Arnoldi decomposition
with the bidiagonal decomposition we used in \cref{sec:projected_newton_method} the method can
be adapted to non-square matrices.

The update for the regularization parameter is done based on the regularized and
the non-regularized residual. Let, in the $k$th iteration, $z_k$ be the solution
without regularization, i.e. $\alpha = 0$, and $y_k$ the solution with the
current best regularization parameter, i.e. $\alpha = \alpha_{k - 1}$. We can then update the regularization parameter using
\begin{equation}\label{eq:aup}
		\alpha_k = \left|\frac{\eta \varepsilon - r(z_k)}{r(y_k) - r(z_k)}\right|\alpha_{k - 1}.
\end{equation}
where $r(z_k) = ||B_{k+1,k} z_k - c_{k+1}||$ and $r(y_k) = ||B_{k+1,k} y_k - c_{k+1}|| $ are the residuals.
A brief sketch of this method is given in \cref{alg:gbit}, but for more information we
refer to \cite{gazzola2014_1, gazzola2014_2, gazzola2014_3}. Note that in the
original GAT method, the non-regularized iterates $z_k$ are equivalent to the
GMRES \cite{saad1986gmres} iterations for the solution of $Ax = b$. Now, because the Arnoldi
decomposition is replaced with the bidiagonal decomposition, they are equivalent
to the LSQR iterations for the solution of $Ax = b$.

\begin{remark}
We use the same stopping criterion in \cref{alg:gbit} as in \cref{alg:PNTM} for a fair comparison between both methods. Note however that since GBiT is based on the standard formulation of the Tikhonov problem \eqref{eq:min_tikhonov}, we have to invert the parameter $\alpha$. Moreover, evaluating this norm would require two additional matrix vector products in each iteration, which is a computationally expensive addition to the algorithm. In an actual implementation GBiT would use a different stopping criterion. 
\end{remark}

\begin{remark}
Note the difference between GBiT (and by extension GAT) and Projected Newton. 
While both methods use increasingly larger Krylov subspaces to solve the Tikhonov problem \eqref{eq:min_tikhonov} and calculate a suitable regularization parameter using the discrepancy principle, the way they do so is different. GBiT solves the projected Tikhonov normal equations \cref{eq:lin_tikhonov} in each Krylov subspace using a fixed
regularization parameter and only afterwords updates the regularization parameter for the next Krylov iteration using \eqref{eq:aup}. The Projected Newton method performs a single iteration of Newton to simultaneously update the solution and regularization parameter using \eqref{eq:newtoneq} and then expands the Krylov subspace basis.
\end{remark}

\begin{algorithm}
		\caption{Generalized bidiagonal-Tikhonov (GBiT) \hfill \textbf{Input:} $A$,$b$,$\alpha_0, \sigma$, tol}\label{alg:gbit}
		\begin{algorithmic}[1]
				\STATE{$x_0 = 0; \quad k = 1; $}
						\WHILE{$\left\|F\left(x_{k-1}, \frac{1}{\alpha_{k-1}}\right)\right\| \geq$ tol}
						\STATE{ Calculate $U_{k + 1}$, $B_{k + 1, k}$ and $V_k$ using Bidiag1}
						\STATE{ Solve $B_{k+1,k}^T B_{k+1,k} z_k = B_{k+1,k}^T c_{k+1}$ for $z_k$.}
						\STATE{ Solve  $(B_{k+1,k}^T B_{k+1,k}+\alpha_{k-1} I_k) y_k = B_{k+1,k}^T c_{k+1}$ for $y_k$.}
						\STATE{ Calculate $\alpha_k$ using \eqref{eq:aup}.}
						\STATE{ $x_k = V_k y_k; \quad k = k + 1;$}
				\ENDWHILE
		\end{algorithmic}
\end{algorithm}

\subsection{Image deblurring} \label{sec:deblurring}

\begin{figure}
\begin{center}
\includegraphics[width=1\textwidth]{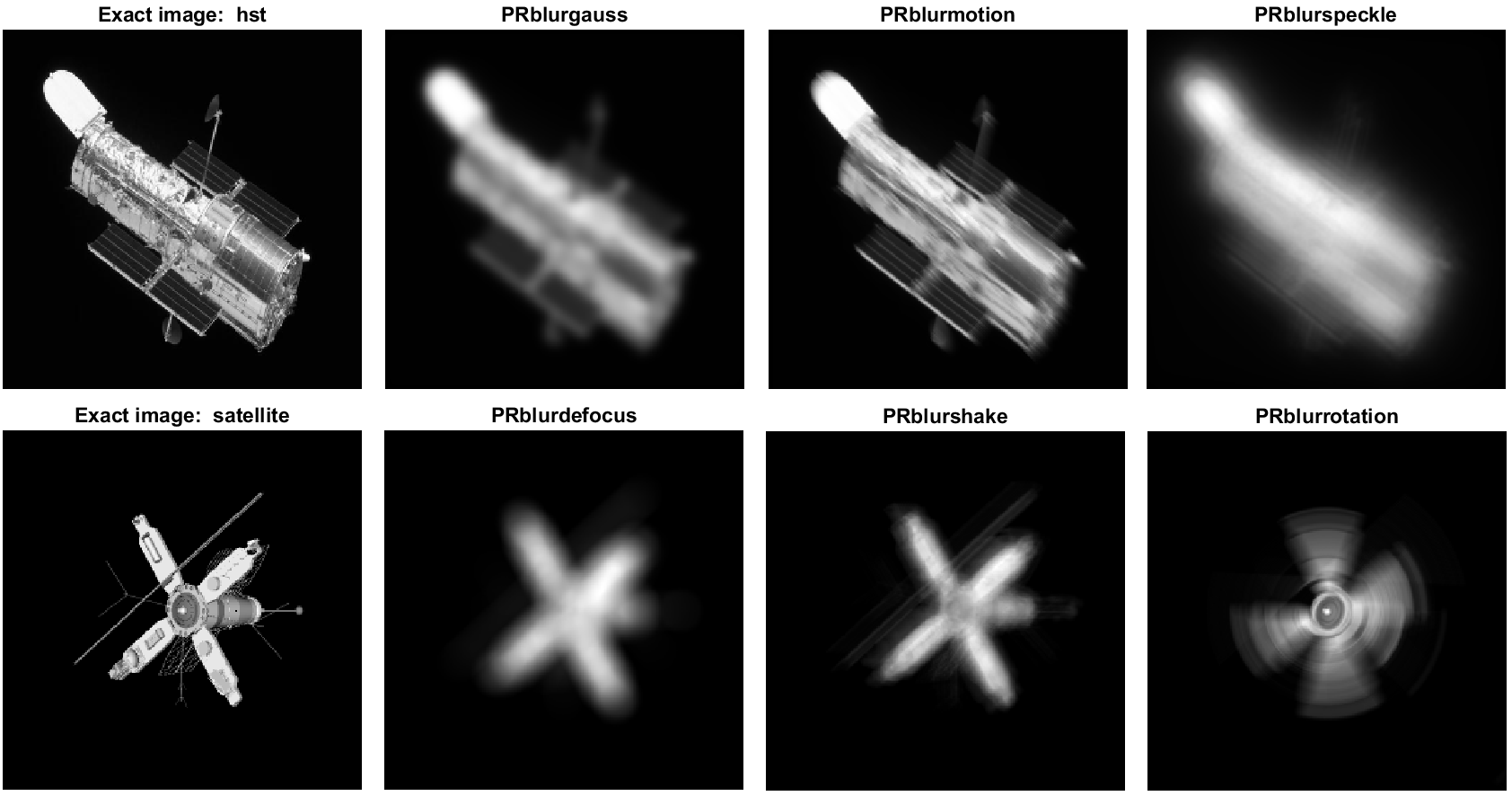}
\end{center}
\caption{The images ``\texttt{hst}'' and ``\texttt{satellite}'', which represent the exact solution $x_{ex}$, together with the distorted images, which represent the exact (i.e. noise-free) right-hand side $b_{ex}$, generated by the different blurring functions from the IR tools package \cite{gazzola2019ir}. } \label{fig:blur}
\end{figure}

Image deblurring is a rich source of linear inverse problems. For example in astronomy, when a ground-based telescope takes a picture of an object in space, the image is typically blurred due to the rapidly changing index of refraction of the atmosphere. Extraterrestrial photographs taken of earth are typically degraded by motion blur due to the slow camera shutter speed relative to the fast spacecraft motion \cite{banham1997digital}. A post-processing phase is then necessary to improve the quality of the picture. Other examples include microscopy, crowd surveillance, positon emmision tomography and many more, see for instance \cite{katsaggelos2012digital,bertero1998introduction,hansen2006deblurring}. 

We use the matlab package IR tools \cite{gazzola2019ir} for generating test problems. This package contains several functions that generate  a matrix $A$, which models the blurring operator in different scenarios, and a corresponding right-hand side $b_{ex}$, which is a distorted version of the exact image. The function \texttt{PRblurrotation}, for example, generates data for an image deblurring problem where the blur simulates rotational motion blur. We consider the functions \texttt{PRblurgauss}, \texttt{PRblurmotion} and \texttt{PRblurspeckle} applied to the exact image ``\texttt{hst}'' and the functions \texttt{PRblurdefocus}, \texttt{PRblurshake} and \texttt{PRblurrotation} applied to the image ``\texttt{satellite}'', see \cref{fig:blur}. These figures are also part of the IR tools package. For more information we refer the reader to \cite{gazzola2019ir}. 

We solve the deblurring test problems described above using GBiT, the Lagrange method and the Projected Newton method with tolerance $10^{-8}$. We apply reorthogonalization to the bidiagonalization procedure in both GBiT and Projected Newton. We set  
$\alpha_0 = \lambda_0 = 1$ as initial regularization parameter and set the maximum number of iterations to $500$. We take test images of size $256 \times 256$ and add $10\%$ Gaussian noise to the right-hand side $b_{ex}$. For simplicity, we solve the Newton system \cref{eq:newton_system} for the Lagrange method with a fixed precision of $10^{-6}$ using the Krylov subspace method MINRES and put the maximum number of iterations of MINRES equal to $100$. The results of the experiment can be found in the top half of \cref{table}. While the number of Newton iterations for the Lagrange method is quite small, the total number of matrix vector product is large since each Krylov iteration requires a matrix vector product with $A$ and $A^T$. Moreover, the backtracking line search also requires two matrix vector products each time the step-length is reduced. It can also be observed the the number of iterations for GBiT and Projected Newton seem to be similar. Note that both methods require two matrix vector products in each iteration and one matrix vector product for initialization. This assumes that we also use the projected function to check for convergence in GBiT, otherwise we get an additional two matrix vector products each iteration. See \cref{fig:blurconvergence} for an example of the convergence history of all three methods in function of the number of matrix vector products. 

\begin{table}
\scalebox{0.95}{
\begin{tabular}{l||rr|rrr|rr|rr}
						 &&& \multicolumn{3}{c|}{Lagrange} & \multicolumn{2}{c|}{GBiT} & \multicolumn{2}{c}{PN} \\
						 \textbf{Problem} & $m$ & $n$ & $\#N$ & $\#\bar{K}$ & \textbf{MV} & $\#K$ &  \textbf{MV} & $\#K$ &  \textbf{MV} \\ \hline \hline
						\texttt{PRblurgauss} & 65,536 & 65,536  & 11&38.8 & 876& 100 & 201 &100& 201 \\
						\texttt{PRblurmotion} & 65,536  & 65,536   &8& 26 & 434& 54 & 109& 54& 109 \\
						\texttt{PRblurspeckle} & 65,536  & 65,536   &11& 38.6 & 876& 86 & 173& 86& 173 \\
						\texttt{PRblurdefocus} &65,536  & 65,536  & 12& 56.8 & 1396& 133 & 267& 132& 265 \\
						\texttt{PRblurshake} & 65,536 & 65,536  &9&28.2 & 528 & 71& 143& 71 & 143 \\
						\texttt{PRblurrotation}& 65,536  & 65,536  &10 & 39.1 & 804& 98 & 197 & 98 & 197 \\ \hline \hline
						\texttt{shepp128} & 23,040 & 16,384 & 14 & 56 & 1666& 50 & 101& 50& 101 \\
						\texttt{shepp256} & 92,160 & 65,536 &16 & 56.2 & 1936& 54 & 109& 54& 109 \\
						\texttt{grains128} & 23,040 & 16,384 &13 & 43 & 1173& 36 & 73& 32& 65 \\
						\texttt{grains256} & 92,160 & 65,536 &15& 47 & 1490 & 51 & 103& 38& 77 \\ \hline \hline
\end{tabular} }
\\
\caption{Results of the numerical experiments as explained in \cref{sec:deblurring} (top) and \cref{sec:ct} (bottom). For the Lagrange method the column $\#N$ denotes the number of Newton iterations, while $\#\bar{K}$ denotes the average number of Krylov iterations per Newton iteration. For GBiT and the Projected Newton method (PN) $\#K$ denotes the number of Krylov subspace iterations. The column \textbf{MV} gives the total number of matrix vector products for each of the methods. This includes the matrix vector products necessary for the backtracking line search in the Lagrange method.} \label{table}
\end{table}

\begin{figure}
\begin{center}
\includegraphics[width=1\textwidth]{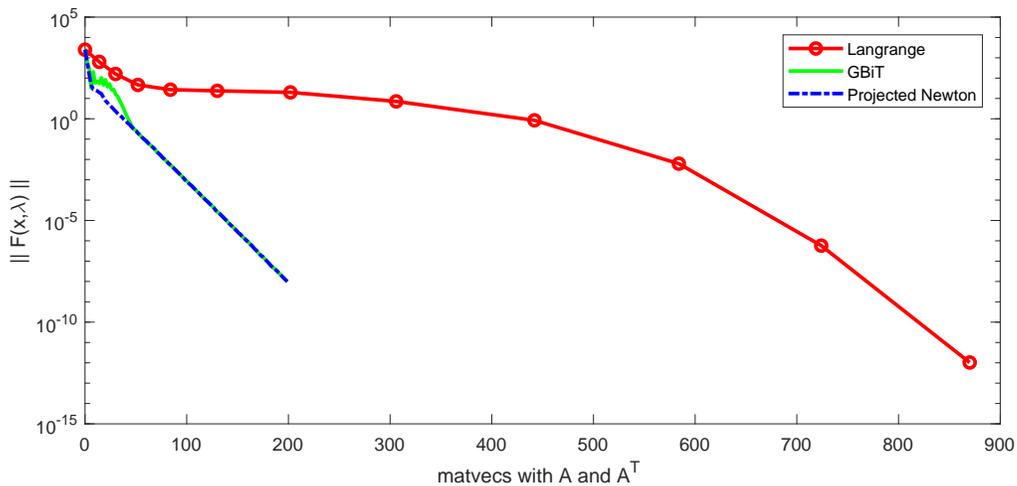}
\end{center}
\caption{Convergence history of \texttt{PRblurgauss} for the Lagrange method, Projected Newton and GBiT in function of the number of matrix vector products. For the Lagrange method, the circles denote the Newton iterations.\label{fig:blurconvergence}}
\end{figure}

\subsection{Computed tomography} \label{sec:ct}

\begin{figure}
\begin{center}
\begin{tabular}{cc}
\includegraphics[width=0.6\textwidth]{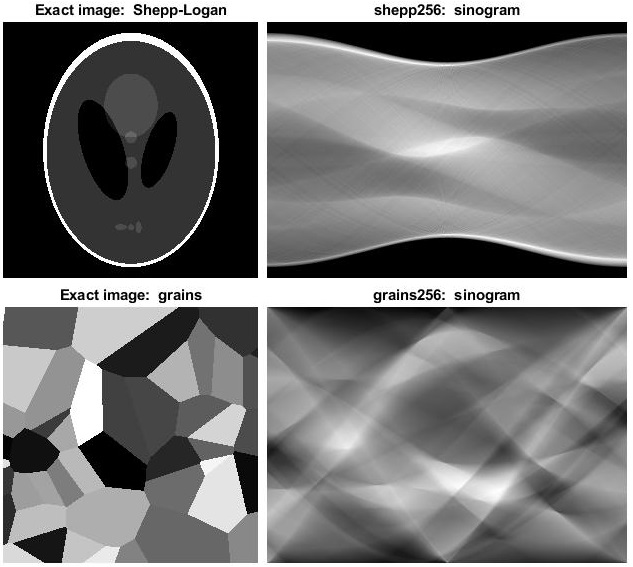}
\end{tabular}
\end{center}
\caption{The modified Shepp--Logan phantom (top) and grains image (bottom) of size $256 \times 256$ as exact solution $x_{ex}$ \cite{hansen2012air,hansen2018air} and corresponding sinogram $b_{ex}$ with $360$ projection angles in $[0, \pi[$.\label{fig:ct} }
\end{figure}

As a second class of test problems we consider x-ray computed tomography \cite{hsieh2009computed}. Here,
the goal is to reconstruct the attenuation factor of an object based on the loss
of intensity in the x-rays after they passed through the object. Classically, the
reconstruction is done using analytical methods based on the Fourier and Radon
transformations \cite{mallat2009}. In the last decades interest has grown in algebraic reconstruction
methods due to their flexibility when it comes to incorporating prior knowledge and
handling limited data. Here, the problem is written as a linear system $Ax = b$, where
$x$ represents the attenuation of the object in each pixel, the right-hand side $b$
is related to the intensity measurements of the x-rays and $A$ is a projection matrix.
The precise structure of $A$ depends on the experimental set-up, but it is typically
very sparse. For more information we refer to \cite{joseph1982, hansen2010, siltanen2012}.

For our computed tomography experiments we consider a parallel beam geometry and use the ASTRA toolbox
\cite{aarle2015, aarle2016} in order to generate the projection matrix $A$. To generate the test images we use the AIR tools package \cite{hansen2012air,hansen2018air}.
As a first test problem we take the modified Shepp--Logan phantom of size $128 \times 128$
and take $180$ projection angles in $[0, \pi[$, which corresponds to a matrix $A$
of size $23,040\times 16,384$. We consider the same experimental set-up but with the Shepp--Logan phantom of size  $256 \times 256$ and with $360$ projection angles  in $[0, \pi[$. We denote these two test problems as \texttt{shepp128} and \texttt{shepp256} respectively. For a third and fourth test problem we take the image \texttt{grains} as exact solution and again consider problem sizes $128 \times 128$ and $256\times 256$ with $180$ and $360$ projection angles respectively. We denote these problems as \texttt{grains128} and \texttt{grains256}. The exact solution $x_{ex}$  and the exact (noise-free) right-hand side $b_{ex}$, typically called a sinogram in computed tomography, for the problems \texttt{shepp256} and \texttt{grains256} are shown in \cref{fig:ct}. We again solve these test problems with the Lagrange method, Projected Newton and GBiT and use the same parameters as in \cref{sec:deblurring}. The results of the experiment are shown in the bottom half of \cref{table}. We can observe that the Lagrange method is not competitive compared to the other two algorithms in terms of matrix vector products. For the problems with the Shepp--Logan phantom we have the same number of Krylov iterations for GBiT and Projected Newton. However, for the grains test problems, the latter algorithm slightly outperforms the former. We investigate the number of Krylov iterations for these methods in a bit more detail in \cref{sec:suitesparse}.

\begin{figure}
\begin{center}
\begin{tabular}{cc}
\includegraphics[width=0.49\textwidth]{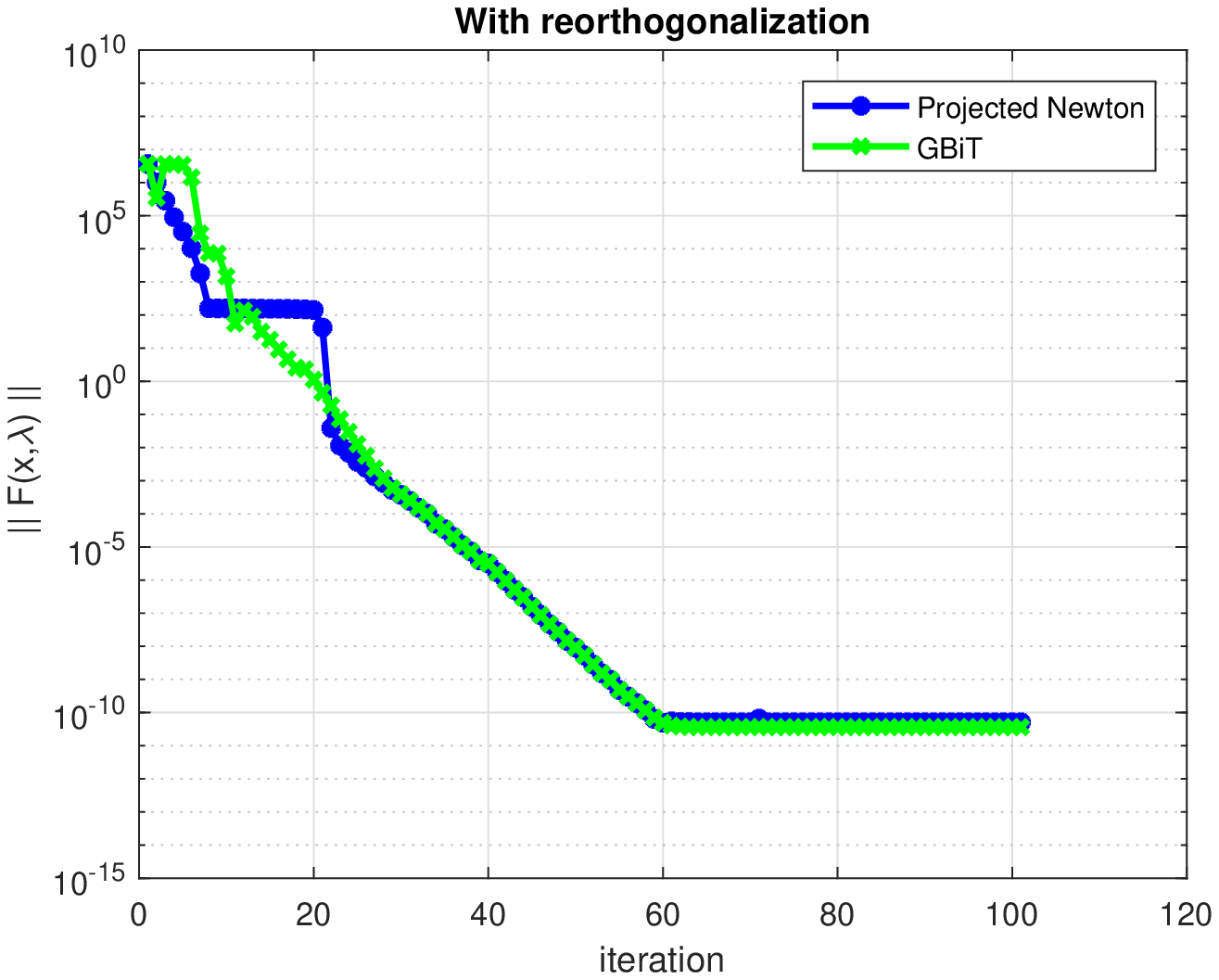} & \hspace{-1cm} \includegraphics[width=0.49\textwidth]{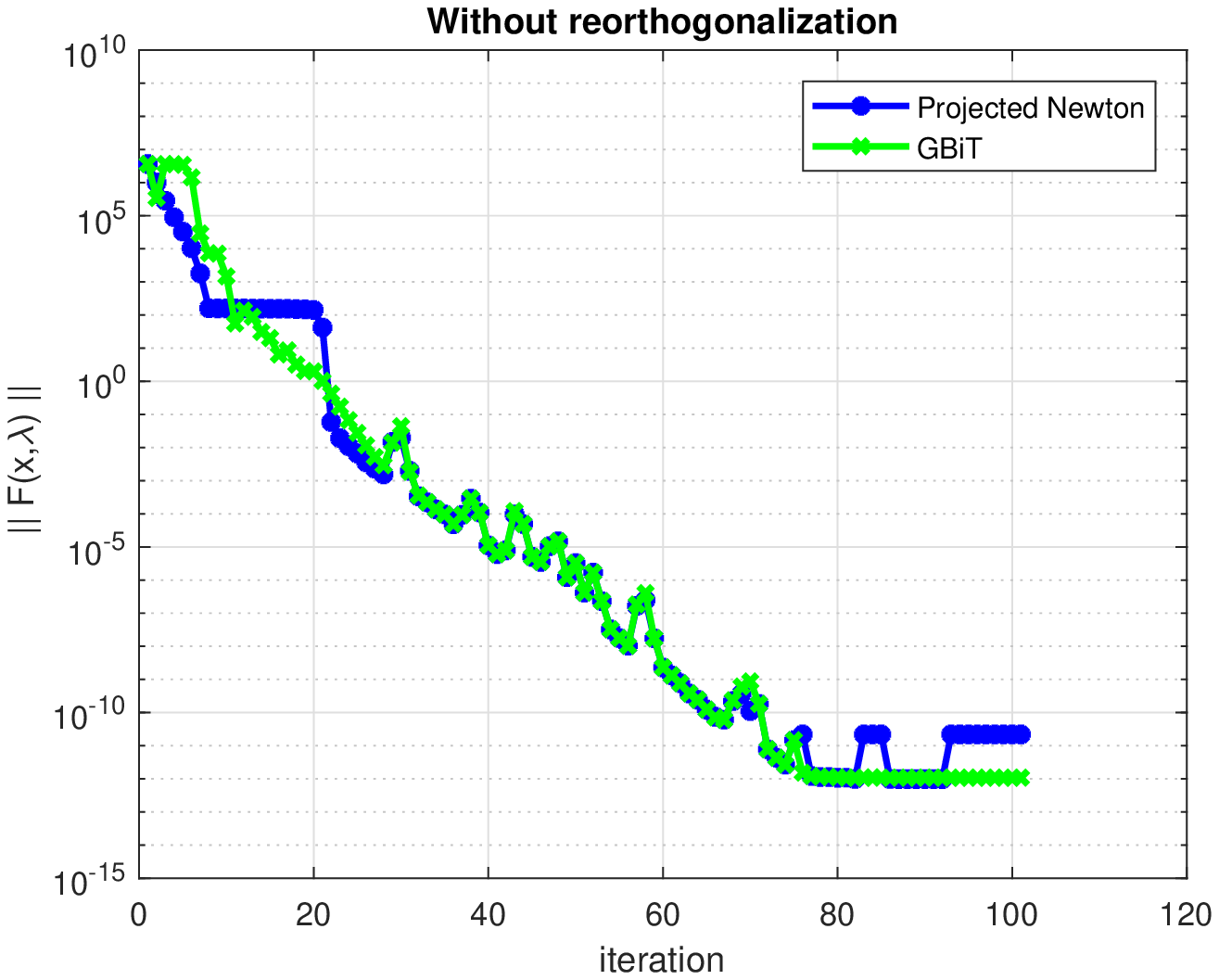} \\
\end{tabular}
\end{center}
\caption{Convergence history of GBiT and Projected Newton method applied to the \texttt{shepp128} test problem. We compare convergence with reorthogonalization (left) and without reorthogonalization (right) of the Krylov subspace bases $V$ and $U$ in each iteration.} \label{fig:reortho}
\end{figure}

It is well known that the Krylov subspace bases $V_k$ and $U_{k+1}$ are not guaranteed to be perfectly orthogonal (i.e. up to machine precision) if we apply the Bidiag1 algorithm in finite precision arithmetic without reorthogonalization \cite{paige1976error,parlett1979lanczos,larsen1998lanczos}.  However, our derivation of the Projected Newton method and proof of convergence heavily rely on the fact that they are. Hence, we are interested in the effect loss of orthogonality of the computed basis vectors has on convergence. To study this effect we solve test problem $\texttt{shepp128}$ using GBiT and Projected Newton with and without reorthogonalization. We use the same parameters as before but set the tolerance well below machine precision to check the attainable accuracy and set the maximum number of iterations to 100. The result is shown in \cref{fig:reortho}.

In the first few iterations the effect of not reorthogonalizing in unnoticeable since at that point the basis vectors are still relatively orthogonal. However, from iteration $28$ onward we can clearly see the difference. The effect is quite similar for GBiT and Projected Newton: while the left plot shows a decreasing series $||F(x_k,\lambda_k)||$, this value increases in some of the iterations on the right plot. Note that this behavior for Projected Newton in the right plot is not possible in exact arithmetic. Indeed, the backtracking line search in the Projected Newton method, see line \ref{lineback} in \cref{alg:PNTM}, ensures that \cref{eq:sufficient_decrease} holds, which means $||F(x_k,\lambda_k)|| < || F(x_{k-1},y_{k-1})||$ for all $k$. However, while this irregular behavior causes a small delay in convergence, it has hardly any effect on the attainable accuracy for this particular experiment. GBiT and Projected Newton with reorthogonalization both reach a tolerance of $10^{-10}$ in $59$ iterations, while it takes them both $66$ iterations when no reorthogonalization is applied. The effect might be more pronounced for other linear inverse problems, but a more thorough analysis of the loss of orthogonality is left as future work.  

\begin{remark}
When the number of Krylov iterations is small the computational overhead of reorthogonalizing the basis vectors is rather limited. However, if we need to perform a lot of Krylov iterations to converge, the additional cost is non-negligible. To reorthogonalize the bases $V_{k+1}$ and $U_{k+1}$ in iteration $k$ we need to calculate $2k$ dot-products, $2k$ multiplications of a vector with a scalar and $2k$ vector additions with vectors of length $n$. This amount to an aditional $\mathcal{O}(8kn)$ floating point operations. We could use more sophisticated techniques like partial reothogonalization to reduce the computation overhead. We refer the reader to \cite{larsen1998lanczos} for more information.  
\end{remark}

\subsection{SuiteSparse Matrix Collection}\label{sec:suitesparse}

\begin{figure}
\begin{center}
\includegraphics[width=1\textwidth]{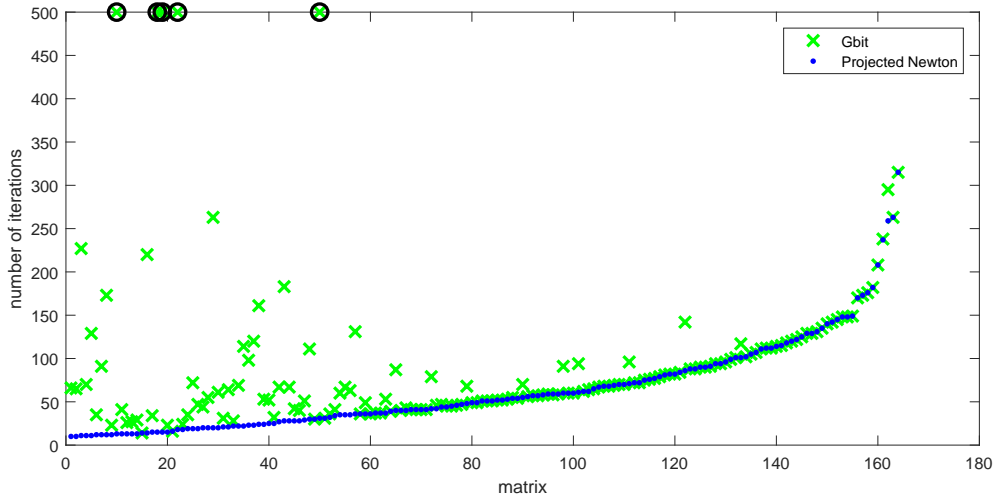}
\caption{Number of iterations (y-axis) needed for GBiT and Projected Newton to converge to the solution of the inverse problem with matrix (x-axis)  from the SuiteSparse Matrix Collection and exact solution $x_{ex, i} = \sin(ih)$ for $h = 2\pi/(n + 1)$. Tolerance for convergence is set to $10^{-8}$ and $10\%$ Gaussian noise is added. The initial regularization parameter is set to  $1/\lambda_0 = \alpha_0 = 10^{-5}$ and the maximum number of iterations is $500$. Black circles indicate when the method did not converge to the desired tolerance within the maximum number of iterations.\label{fig:number_of_its}}
\end{center}
\end{figure}


\begin{figure}
\begin{center}
\begin{tabular}{cc}
\includegraphics[width=0.48\textwidth]{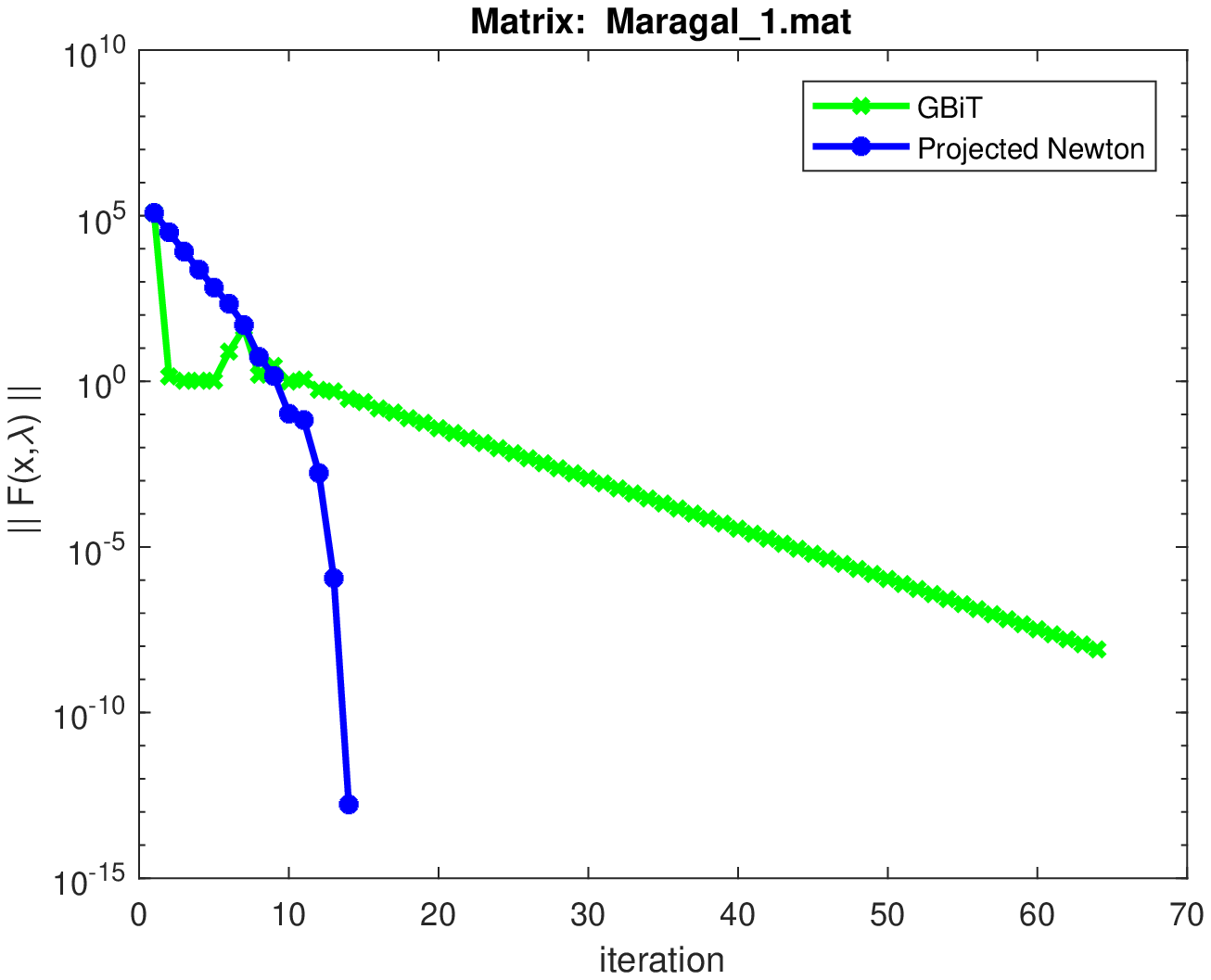} & \hspace{-1cm} \includegraphics[width=0.48\textwidth]{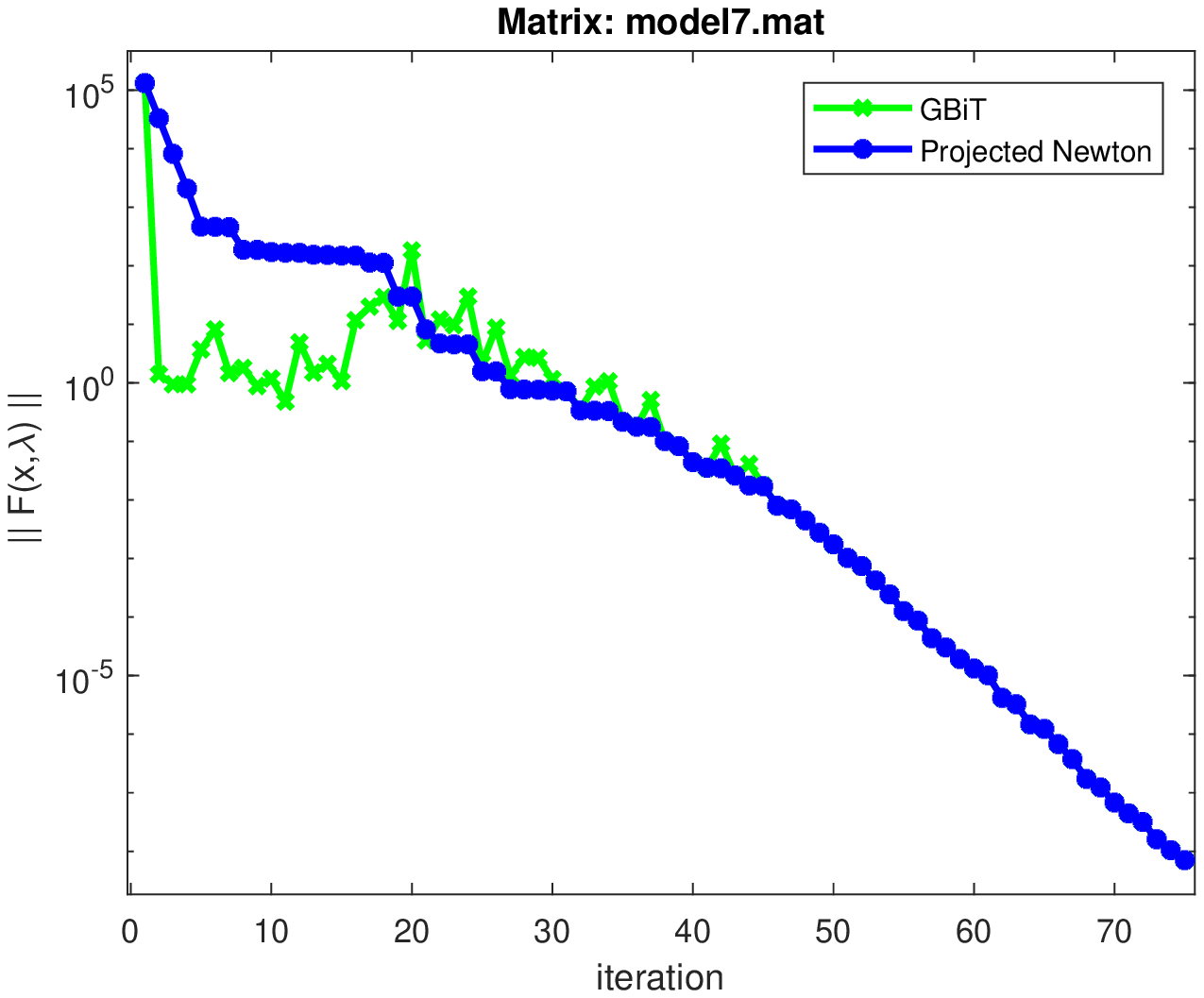} \\
\includegraphics[width=0.48\textwidth]{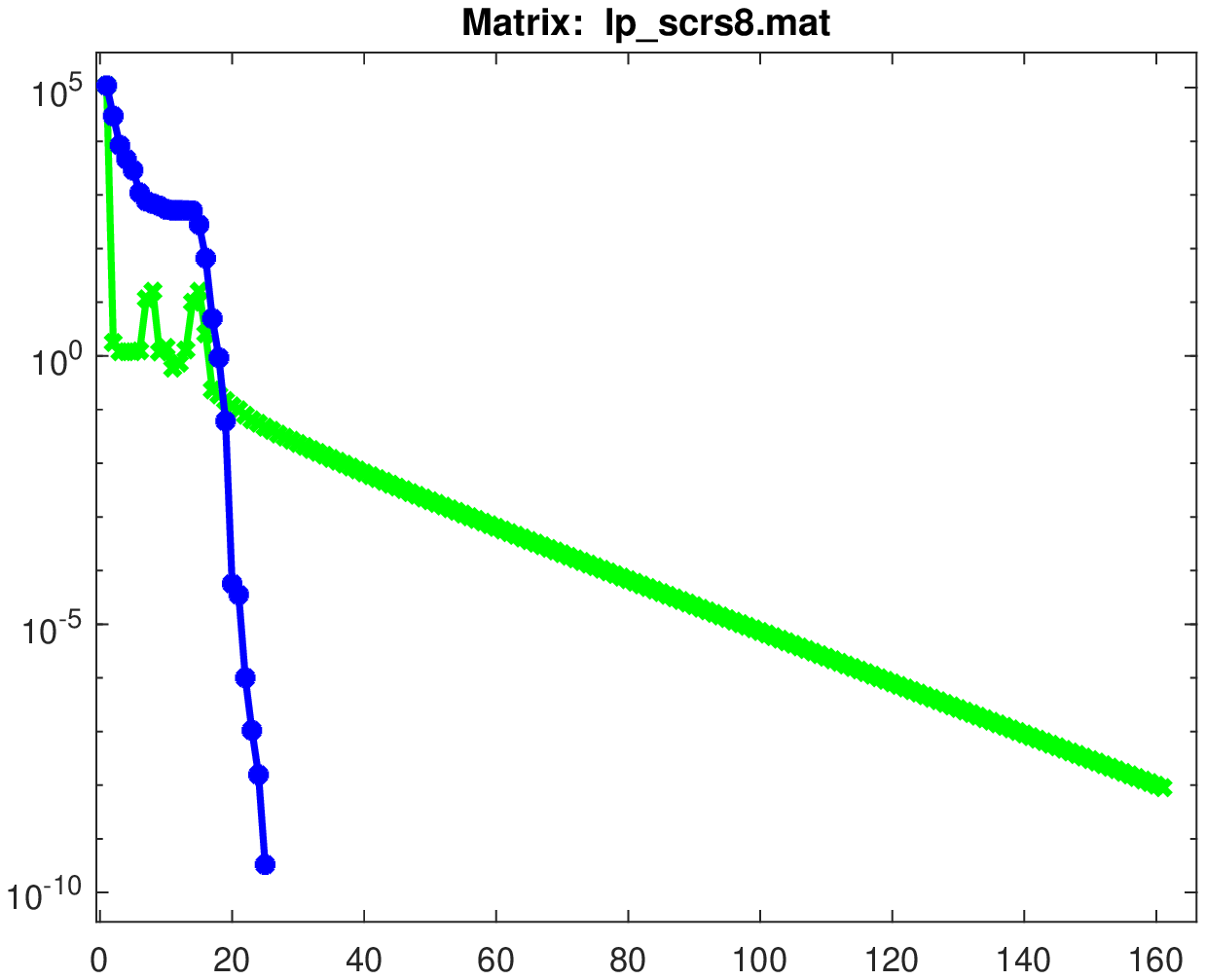} & \hspace{-1cm} \includegraphics[width=0.48\textwidth]{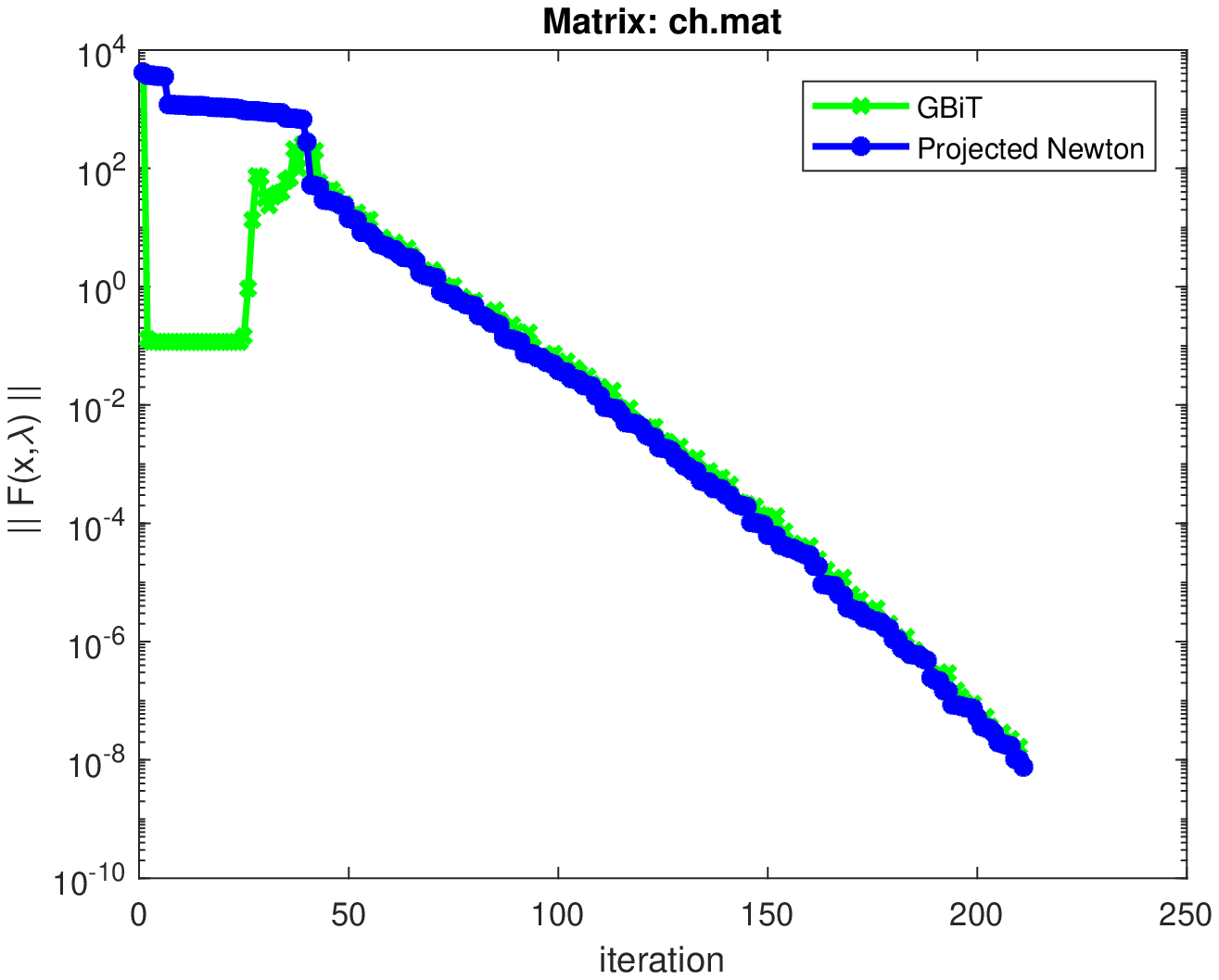}  
\end{tabular}
\end{center}
\caption{Convergence history of GBiT and Projected Newton method applied to four selected matrices from the SuiteSparse Matrix Collection, see \cref{sec:suitesparse}. Left: two examples where Projected Newton significantly outperforms GBiT, due to the quadratic (local) convergence rate of Newton's method. Right: two examples where convergence of GBiT and Projected Newton is very similar, since the convergence rate is determined by the dimension of the Krylov subspace.\label{fig:convergencerate}} 
\end{figure}

As a final experiment we compare the number of Krylov subspace iterations needed for GBiT and Projected Newton to converge. We leave out the Lagrange method in this experiment since it is clearly not competitive with the Krylov subspace based approaches, see \cref{table}. 
We selected all real valued rectangular matrices from the SuiteSparse Matrix Collection \cite{davis2011} with number of rows and columns less than $10,000$, resulting in a total of 164 matrices $A\in\mbbR^{m\times n}$. When $m < n$ we take the transpose of the matrix. Next, we normalize the matrix such that the problem is well-scaled and then we generate a solution vector $x_{ex}\in\mbbR^n$ with entries $x_{ex, i} = \sin(ih)$ for $h = 2\pi/(n + 1)$ and $1\leq i \leq n$, calculate the right-hand side $b_{ex} = Ax_{ex}\in\mbbR^m$ and add $10\%$ Gaussian noise.  We compare the total number of iterations that GBiT and Projected Newton need to converge to the solution with tolerance tol $=10^{-8}$. We use $1/\lambda_0 = \alpha_0 = 10^{-5}$ for all problems and put the maximum number of iterations equal to $500$. Both methods again reorthogonalize the Krylov subspace bases in each iteration. 

The results of the experiment are given by \cref{fig:number_of_its}. We have ordered the $164$ matrices by the number of iterations needed for Projected Newton to converge. First note that the Projected Newton method converges for each of the $164$ matrices within the maximum number of iteration, while GBiT does not converge to the desired tolerance for five of the matrices. Moreover, the number of iterations for Projected Newton to converge for these particular problems is  less than or equal to the number of iterations of GBiT. Note however that for $91$ of the $164$ matrices, which is approximately $55\%$, the number of iterations for both methods is exactly the same. An intuitive explanation for this observation is the following: the rate of convergence of both methods is either limited by the dimension of the Krylov subspace, in which case both methods behave quite similarly, or convergence is determined by the rate of convergence of the root-finder. In the latter case, Projected Newton outperforms GBiT, since Newton's method has a quadratic (local) convergence while the secant method has only a linear rate of convergence. We illustrate this hypothesis with a few representative examples, see \cref{fig:convergencerate}.

\begin{figure}
\begin{center}
\begin{tabular}{cc}
\includegraphics[width=0.49\textwidth]{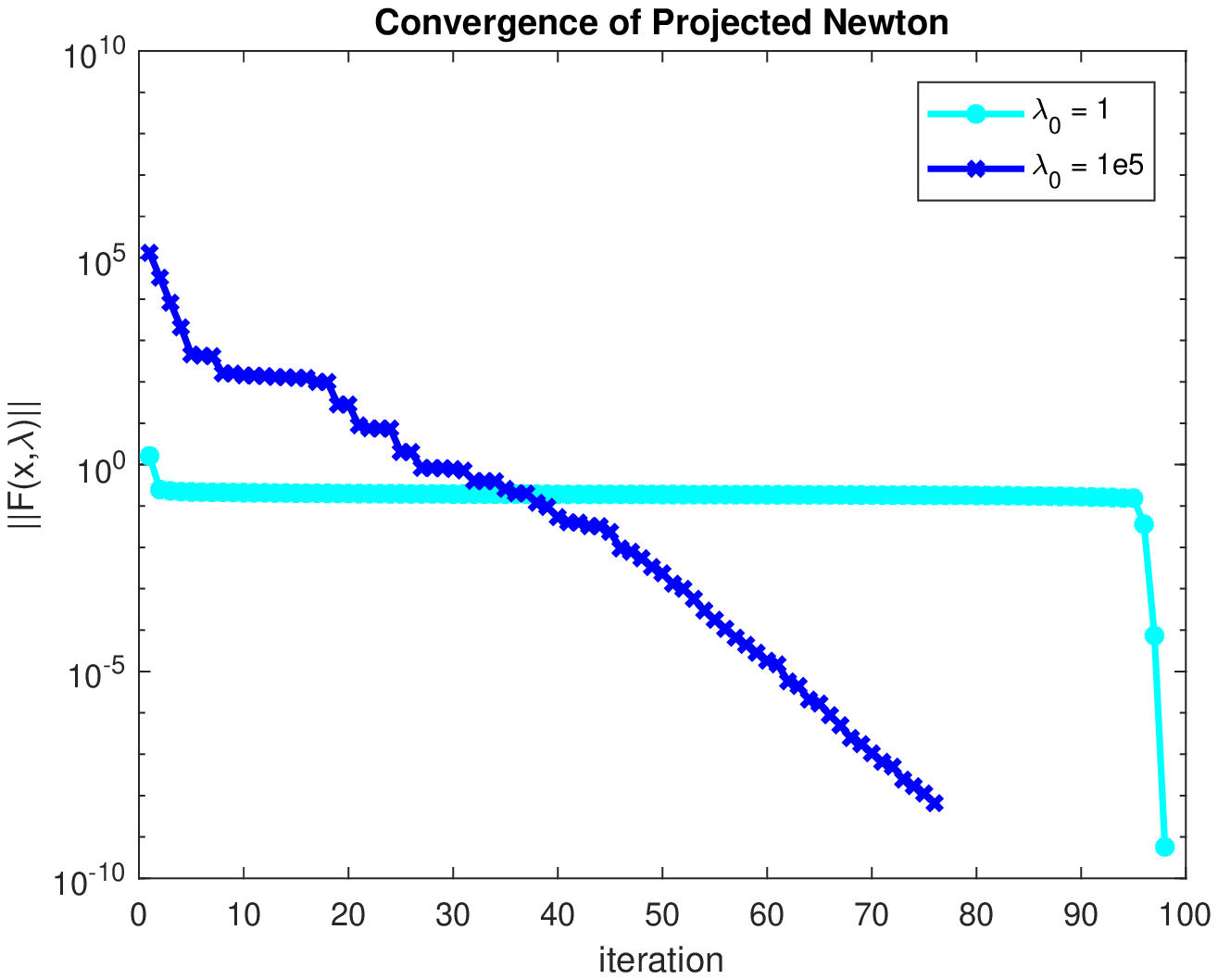} & \includegraphics[width=0.49\textwidth]{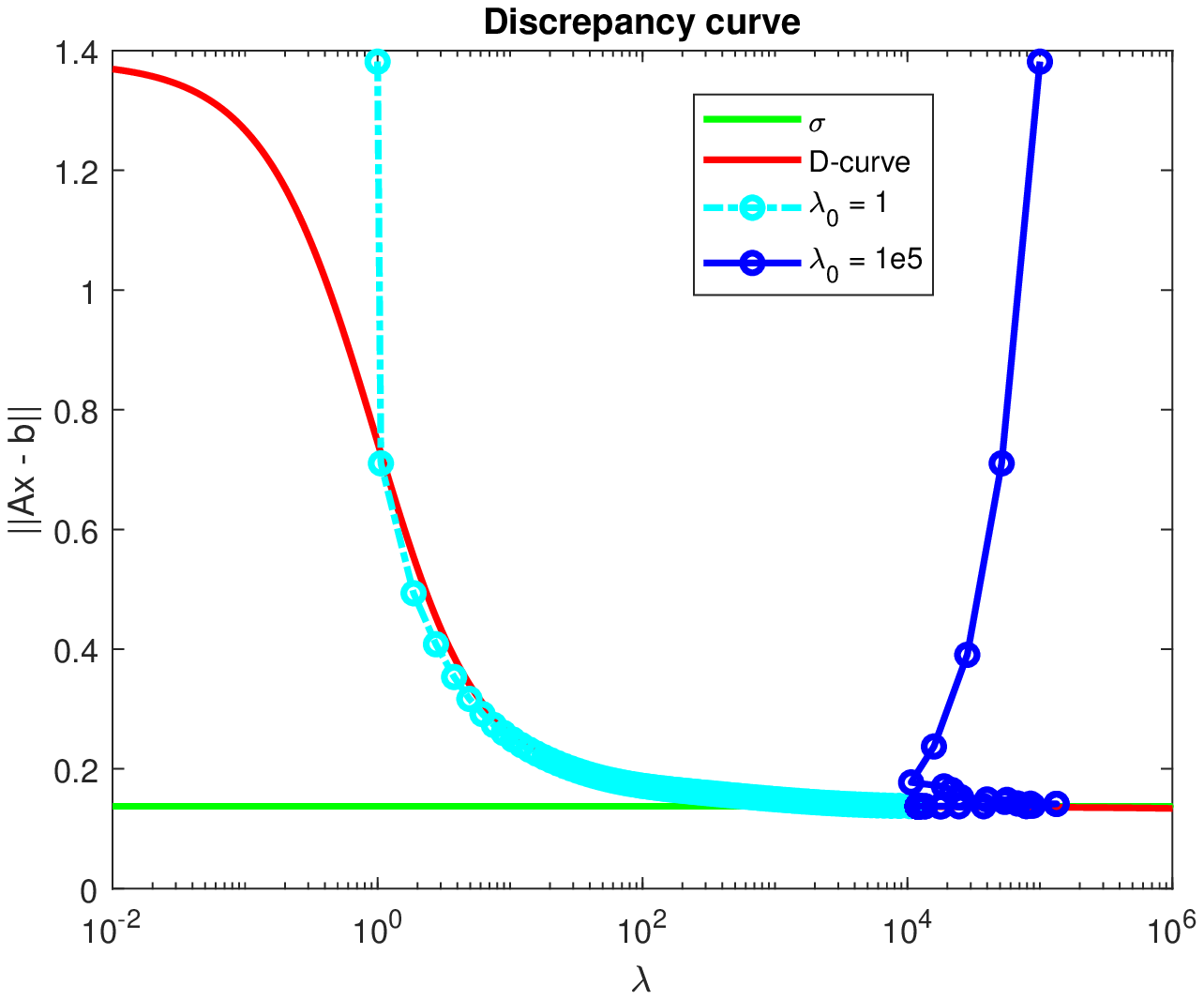} 
\end{tabular}
\end{center}
\caption{(Left) Convergence history of Projected Newton applied to the matrix ``model7'' from the SuiteSparse matrix collection for two different choices of initial regularization parameter $\lambda_0=1$ and $\lambda_0 = 10^5$, see \cref{sec:suitesparse}. (Right) The discrepancy curve $(\lambda,||Ax_\lambda - b||)$ (red) and the value $\sigma$ used in the discrepancy principle (green). The figure shows the values $\lambda_k$ and $||Ax_k - b||$ of the iterates in the Projected Newton method for both choices of initial parameter. \label{fig:dcurve}}
\end{figure}

Lastly, we motivate our choice of initial regularization parameter $\lambda_0 = 10^{5}$ used in the experiments of the current section.
To illustrate the effect this parameter has on the convergence of Projected Newton, we select a specific matrix from the SuiteSparse Matrix collection, namely the matrix named ``model7'' and solve the problem once using $\lambda_0=1$ and again using $\lambda_0=10^{5}$. The other parameters are kept the same as described above. As can be observed from \cref{fig:dcurve}, the behavior is quite different for the two choices: convergence is much slower for $\lambda_0 = 1$ and there is a large number of iterations where the value $||F(x,\lambda)||$ is almost constant.
This can be explained by inspection of the curve $(\lambda,||Ax_\lambda - b||)$, where $x_\lambda$ satisfies $\lambda A^T(Ax_\lambda - b) + x_\lambda = 0$. By considering the transformation of variables $\lambda = 1/\alpha$ as in (\ref{eq:transf}), it is easy to see that this is actually the discrepancy curve, see \cref{fig:curves}. It is well known that convergence of Newton's method is slow when the Jacobian matrix is ill-conditioned \cite{nocedal2006numerical}. If we denote $d(\lambda) = \frac{1}{2}||Ax_\lambda - b||^2$, then it follows from a small calculation that $d'(\lambda) = -(Ax_\lambda - b)^TA(\lambda A^T A + I) A^T(Ax_\lambda - b)$. This value is precisely the Schur complement of the Jacobian matrix \cref{eq:jacf} and is a factor in the expression of its determinant \cref{eq:determinant}. This allows us to intuitively explain the observed behavior. If we choose $\lambda_0=1$, the iterates of the Projected Newton method closely follow the discrepancy curve. Now if the discrepancy curve has points that are close to stationary, as is the case in the example shown in \cref{fig:dcurve}, then $d'(\lambda)$ is close to zero, which is turn means that the Jacobian matrix is ill-conditioned. Hence, we can expect a lower quality of descent direction in that case. Note that we hand-picked this example such that the solution of noise constrained Tikhonov problem \cref{eq:noise_constrained} is a point on the discrepancy curve that is nearly stationary to be able to explain the influence of the initial regularization parameter on convergence. In many of the performed experiments this effect was much less pronounced. However, to assure the best possible performance of the Projected Newton method, we suggest to use a large initial regularization parameter. Note that the initial parameter $\alpha_0$ has hardly any effect on the performance of GBiT. Indeed, in \cite{gazzola2014_1} an experiment is performed where the same problem is solved used different initial regularization parameters and the number of iterations needed to converge was the same for all parameters. 

\begin{remark}
The experiments in the current section mostly deal with the convergence rate of different algorithms for computing the solution of the noise constrained Tikhonov problem \cref{eq:noise_constrained}, which can be seen as combining Tikhonov regularization with the discrepancy principle, but do not study the quality of the solution. Another popular regularization approach is to combine an iterative method, like Conjugate Gradients applied to the Normal Equations (CGNE) \cite{engl1996regularization}, with early stopping, i.e. terminating the algorithm when the residual norm $||Ax_k - b||$ becomes smaller than $\sigma = \eta\epsilon$, see \cref{eq:discrepancy_principle}. It can be expected that CGNE requires fewer iterations than the Projected Newton method, so it would be interesting to study the quality of the solution of both approaches. One clear advantage of combining Tikhonov regularization with the discrepancy compared to CGNE with early stopping is that if the noise-level was slightly underestimated, the latter would still show semi-convergence and the solution would be much worse than the solution obtained with Projected Newton, which does not suffer from semi-convergence. This is of course also the reason that the discrepancy principle uses a safety factor $\eta>1$. Using the noise constrained Tikhonov formulation allows us to use a more accurate estimate of the noise-level $\epsilon$, since we do not have to worry about underestimating this value.
\end{remark}

\section{Conclusions \& outlook} \label{sec:conclusion}
In this work we develop a new algorithm which simultaneously calculates the regularization parameter and corresponding Tikhonov regularized solution of an ill-posed least squares problem such that the discrepancy principle is satisfied. In \cref{sec:newtons_method} we describe how this problem can be characterized as a constrained optimization problem. By projecting the problem onto a low dimensional Krylov subspace using the bidiagonalization procedure, we obtain a projected optimization problem, for which a Newton direction can be calculated very efficiently. We then show that this projected Newton direction produces a descent direction for the original problem. This result allows us to formulate the Projected Newton algorithm for which we prove a global convergence result. We consider some test problems from image deblurring and computed tomography to show the validity of the approach and compare Projected Newton with two other algorithms, namely the Lagrange method from \cite{landi2008lagrange} and the Generalized bidiagonal-Tikhonov method (GBiT), which uses the same bidiagonalization procedure. A first observation we make is that the Lagrange method is not competitive compared to the Krylov subspace based approaches in terms of the number of matrix vector products with $A$ and $A^T$. Next, we compare the number of Krylov subspace iterations needed for GBiT and Projected Newton to converge to a solution with the same tolerance. While in the majority of the experiments reported both methods roughly perform the same number of iterations, the Projected Newton method significantly outperforms GBiT for a large portion of the inverse problems. We hypothesize that this is due to the fact that, when convergence is not determined by the dimension of the Krylov subspace, the quadratic convergence rate of Newton's method beats the linear rate of convergence of the secant method, which is used in GBiT. 

A first possible future research direction is explained in \cref{remark_general}. Since the Projected Newton method in its current form is only able to solve the general form Tikhonov problem, we are interested in how similar ideas can be used to solve the more general regularization problem \cref{eq:general_reg}. Although this work presents a solid theoretical foundation for the algorithm, some interesting research questions remain unanswered. A more formal discussion of the rate of convergence of the Projected Newton method is desirable. Furthermore, finite precision behavior of the algorithm is also something that would benefit from further investigation. More specifically, we are interested in the importance of the reorthogonalization step in the bidiagonalization algorithm, since the proofs we present rely heavily on the fact that we have an orthonormal basis. It is well known that a loss of orthogonality can be observed if we apply the bidiagonalization procedure without reorthogonalization. A small numerical experiment investigating loss or orthogonality is shown in \cref{sec:ct}, but a formal analysis deserves to be treated as future work. Lastly, we use the discrepancy principle to determine a suitable regularization parameter. However, there exist better a posteriori parameter choice methods for Tikhonov regularization based on the noise-level \cite{engl1996regularization}. Hence, it would be interesting to study if we could use similar techniques as the ones presented in this work to efficiently combine Tikhonov regularization with a different parameter choice rule.    

\section*{Acknowledgments}
We would like to acknowledge the Department of Mathematics and Computer Science, University of Antwerp, for financial support. This work was funded in part by the IOF-SBO project entitled ``High performance iterative reconstruction methods for Talbot Lau grating interferometry based phase contrast tomography". The authors are also grateful to the two anonymous referees for providing useful comments which helped significantly improve the manuscript.

\section*{References}
\bibliographystyle{unsrt}
\bibliography{references}

\end{document}

%% file: curves.tex
\begin{tikzpicture}[line cap = round, line join = round, > = triangle 45, scale = 0.65]
		\hspace{-0.4cm}
		\begin{scope}[shift = {(-7, 0)}]
				\draw[->] (-1, 0) -- (5, 0);
				\draw[->] (0, -1) -- (0, 5);
				
				 L-curve:
				\draw [shift = {(0, 3)}, thick]  plot[domain = 0:1.07,variable = \t] ({cos(\t r)},{sin(\t r)});
				\draw [shift = {(3, 0)}, thick]  plot[domain = 0.5:1.57,variable = \t] ({cos(\t r)},{sin(\t r)});
				\draw [shift = {(1.5, 1.5)}, thick]  plot[domain = 3.14:4.71,variable = \t] ({cos(\t r)/2},{sin(\t r)/2});
				\draw [thick] (1, 3) -- (1, 1.5);
				\draw [thick] (1.5, 1) -- (3, 1);
				
				\draw [dashed, thick, color = red] (1.75, -0.25) -- (1.75, 3.25);
				\node [anchor = east] at (1.6, 0.35) {\color{red} $\eta\varepsilon$};
				
				\draw [shift = {(0, 3)}] plot[domain = 0.32:1.25, variable = \t] ({1*1.58*cos(\t r)+0*1.58*sin(\t r)},{0*1.58*cos(\t r)+1*1.58*sin(\t r)});
				\draw [shift = {(3, 0)}] plot[domain = 0.32:1.25, variable  =\t] ({1*1.58*cos(\t r)+0*1.58*sin(\t r)},{0*1.58*cos(\t r)+1*1.58*sin(\t r)});
				\draw [->] (0.5, 4.5) -- (0.29, 4.57);
				\draw [->] (4.5, 0.5) -- (4.57, 0.3);
				
				\node [label = below:$\log \left\|Ax_\alpha - b\right\|$] at (2.5, -0.25) {};
				\node [label = left:\rotatebox{90}{$\log \left\|x_\alpha\right\|$}] at (-0.25, 2.5) {};
				\node [font = \tiny, label = right:$0\leftarrow\alpha$] at (0.75, 4.5) {};
				\node [font = \tiny, label = right:$\alpha\rightarrow+\infty$] at (4, 1.25) {};
				
				
				\draw[->] (2.75, 2.25) -- (1.8, 1.1);
				\node [align = center] at (4, 2.75) {\footnotesize $\alpha$ for discrepancy\\ \footnotesize principle};
				\draw[->] (-1.5, -0.2) -- (1.1, 1.1);
				\node [align = center] at (-1.5, -1) {\footnotesize $\alpha$ for L-curve\\ \footnotesize method};
		\end{scope}
		\hspace{0.7cm}
		\begin{scope}[shift = {(1, 0)}]
				\draw[->] (-1, 0) -- (7, 0);
				\draw[->] (0, -1) -- (0, 5);
				
				\draw [thick] (0, 0.25) .. controls (2, 0.25) and (4, 2) .. (6, 4);
				
				\draw [dashed, thick, color = red] (-0.1, 2) -- (7, 2);
				\node [anchor = north] at (7, 1.9) {\color{red} $\eta\varepsilon$};
				
				\node [anchor = north] at (3.5, 0) {$\alpha$};
				\node [anchor = east] at (-0.1, 2.5) {\rotatebox{90}{$\left\|Ax_\alpha - b\right\|$}};
				
				\draw[->] (3.3, 3.1) -- (3.8, 2.1);
				\node [align = center] at (2.75, 3.75) {\footnotesize $\alpha$ for discrepancy\\  \footnotesize principle};
				
				\draw [thick, decorate, decoration={brace, mirror}] (0, -0.6) -- (3.7, -0.6)
						node [midway, yshift = -10] {\footnotesize Overfitting};
				\draw [thick, decorate, decoration={brace, mirror}] (3.8, -0.6) -- (7, -0.6)
						node [midway, yshift = -10] {\footnotesize Oversmoothing};
		\end{scope}
\end{tikzpicture}